%% file: band_infinitesimal.tex
\documentclass[10pt]{amsart}

\usepackage[pdfstartpage={1},pdfstartview={FitH},pdftitle={Finite-rank perturbations of random band matrices via infinitesimal free probability},pdfauthor={Benson Au},pdfsubject={2010 Mathematics Subject Classification: 15B52; 46L53; 46L54; 60B20},pdfcreator={},pdfproducer={},pdfkeywords={Finite-rank perturbation; infinitesimal free probability; random band matrix; traffic probability},pagebackref=false,colorlinks=true,linkcolor={red!87.5!black},citecolor={blue!50!black},linktocpage=true]{hyperref}

\usepackage{accents}
\usepackage{afterpage}
\usepackage{amssymb}
\usepackage{appendix}
\usepackage{array}
\usepackage{bm}
\usepackage{bbm}
\usepackage{braket}
\usepackage{centernot}
\usepackage{cite}
\usepackage{enumitem}
\usepackage{graphicx}
\usepackage{indentfirst}
\usepackage{mathdots}
\usepackage{mathrsfs}
\usepackage{mathtools}
\usepackage{mleftright}
\usepackage{pgf}
\usepackage{scalerel}
\usepackage{stackengine}
\usepackage{tikz}
\usetikzlibrary{arrows}
\allowdisplaybreaks[4]

\MHInternalSyntaxOn
\renewcommand{\dcases}
 {
  \MT_start_cases:nnnn
    {\quad}
    {$\m@th\displaystyle##$\hfil}
    {$\m@th\displaystyle##$\hfil}
    {\lbrace}
 }
\MHInternalSyntaxOff

\newtheorem{thm}{Theorem}[section]

\newtheorem{cor}[thm]{Corollary}

\newtheorem{lemma}[thm]{Lemma}
\newtheorem{prop}[thm]{Proposition}

\theoremstyle{definition}
\newtheorem{defn}[thm]{Definition}

\newtheorem{eg}[thm]{Example}

\theoremstyle{remark}
\newtheorem{rem}[thm]{Remark}

\newcommand{\N}{\mathbb{N}}

\newcommand{\R}{\mathbb{R}}
\newcommand{\C}{\mathbb{C}}

\newcommand{\prob}{\mathbb{P}}
\newcommand{\E}{\mathbb{E}}

\newcommand{\mcal}[1]{\mathcal{#1}}

\newcommand{\pto}{\overset{\prob}{\to}}

\newcommand{\dto}{\overset{d}{\to}}
\newcommand{\deq}{\overset{d}{=}}

\newcommand{\indc}[1]{\mathbbm{1}\{#1\}}

\newcommand{\source}{\operatorname{src}}
\newcommand{\target}{\operatorname{tar}}

\newcommand{\interior}[1]{\accentset{\circ}{#1}}

\newcommand{\floor}[1]{\lfloor #1 \rfloor}

\newcommand{\mbf}[1]{\mathbf{#1}}
\newcommand{\mfk}[1]{\mathfrak{#1}}
\newcommand{\op}[1]{\operatorname{#1}}

\newcommand{\matN}{\op{Mat}_N}

\renewcommand{\Re}{\operatorname{Re}}
\renewcommand{\Im}{\operatorname{Im}}

\DeclareMathOperator{\Tr}{Tr}

\DeclareMathOperator{\Cat}{Cat}

\DeclareRobustCommand{\SkipTocEntry}[5]{} 

\makeatletter
\newcommand{\subalign}[1]{%
  \vcenter{%
    \Let@ \restore@math@cr \default@tag
    \baselineskip\fontdimen10 \scriptfont\tw@
    \advance\baselineskip\fontdimen12 \scriptfont\tw@
    \lineskip\thr@@\fontdimen8 \scriptfont\thr@@
    \lineskiplimit\lineskip
    \ialign{\hfil$\m@th\scriptstyle##$&$\m@th\scriptstyle{}##$\crcr
      #1\crcr
    }%
  }
}
\makeatother

\begin{document}

\author{Benson Au}
\address{University of California, San Diego\\
         Department of Mathematics\\
         9500 Gilman Drive \# 0112\\
         La Jolla, CA 92093-0112\\
         USA}
       \email{\href{mailto:bau@ucsd.edu}{bau@ucsd.edu}}
\date{\today}
\title[Finite-rank perturbations of random band matrices]{Finite-rank perturbations of random band matrices via infinitesimal free probability}


\begin{abstract}\label{abstract}
We prove a sharp $\sqrt{N}$ transition for the infinitesimal distribution of a periodically banded GUE matrix. For band widths $b_N = \Omega(\sqrt{N})$, we further prove that our model is infinitesimally free from the matrix units and the normalized all-ones matrix. Our results allow us to extend previous work of Shlyakhtenko on finite-rank perturbations of Wigner matrices in the infinitesimal framework. For finite-rank perturbations of our model, we find outliers at the classical positions from the deformed Wigner ensemble.
\end{abstract}

\maketitle
\tableofcontents

\section{Introduction}\label{sec:intro}

\subsection{Motivation}\label{sec:motivation}

The contact between random matrices and free probability first appeared in the seminal work of Voiculescu \cite{Voi91}. By now, a well-developed theory exists to illustrate the depth of this connection: see, for example, the monographs \cite{VDN92,NS06,AGZ10,MS17}. We summarize the basic paradigm as follows: in many generic situations, independent random matrices become freely independent in the large $N$ limit. The analytic machinery of free probability then allows us to understand various joint asymptotics associated to such multi-matrix models.

Despite the tremendous success of this approach, the standard free probability framework comes with inherent limitations. In particular, free independence only prescribes the zeroth order behavior of our random variables: for random matrices, this shortcoming already manifests itself at the level of outliers. To make this precise, we introduce some notation. In this article, we restrict our attention to self-adjoint matrices. For such a matrix $\mbf{A}_N \in \matN(\C)$, we write $(\lambda_k(\mbf{A}_N))_{k \in [N]}$ for its eigenvalues, counting multiplicity, arranged in a non-increasing order. We further write $\mu(\mbf{A}_N)$ for the \emph{empirical spectral distribution} (\emph{ESD}) of $\mbf{A}_N$. Thus,
\[
  \lambda_1(\mbf{A}_N) \geq \cdots \geq \lambda_N(\mbf{A}_N), \qquad \mu(\mbf{A}_N) = \frac{1}{N}\sum_{k \in [N]} \delta_{\lambda_k}.
\]
Hereafter, when we refer to a matrix $\mbf{A}_N$, we implicitly refer to a sequence of matrices $(\mbf{A}_N)_{N \in \N}$.

Now, suppose that we have random matrices $\mbf{A}_N$ and $\mbf{B}_N$ such that the ESDs converge weakly in expectation to some compactly supported probability measures $\mu_{\mbf{A}}$ and $\mu_{\mbf{B}}$ respectively.  If we further assume that $\mbf{A}_N$ and $\mbf{B}_N$ are asymptotically free, then we can even compute the limiting spectral distribution (LSD) of rational functions in the pair $(\mbf{A}_N, \mbf{B}_N)$ \cite{HMS18}. In particular, the freeness relationship completely determines the LSD of the sum $\mbf{C}_N = \mbf{A}_N + \mbf{B}_N$ from the marginals $\mu_{\mbf{A}}$ and $\mu_{\mbf{B}}$. By analogy with the classical case, this operation is known as the \emph{free (additive) convolution}, for which we use the notation $\mu_{\mbf{C}} = \mu_{\mbf{A}} \boxplus \mu_{\mbf{B}}$. We recall the following characterization of the free convolution in terms of subordination functions (see, for example, \cite[Chapter 3]{MS17}). For a probability measure $\mu$ on $\R$, we denote its \emph{Cauchy transform} by $G_\mu: \C^+ \to \C^-$, where
\[
G_\mu(z) = \int_\R \frac{1}{z - t}\, \mu(dt).
\]
We use the notation $F_\mu = \frac{1}{G_\mu}: \C^+ \to \C^+$ for the reciprocal Cauchy transform.

\begin{thm}[\hspace{-1sp}\cite{Voi93,Bia98}]\label{thm:subordination}
For any pair of probability measures $\mu_1, \mu_2$ on $\R$, there exists a unique pair of analytic functions $\omega_1, \omega_2: \C^+ \to \C^+$ such that
\begin{enumerate}[label=(\roman*)]
\item $G_{\mu_1}(\omega_1(z)) = G_{\mu_2}(\omega_2(z))$;
\item $\omega_1(z) + \omega_2(z) = z + F_{\mu_1}(\omega_1(z))$.
\end{enumerate}
Moreover, the common function in property (i) corresponds to the Cauchy transform of a unique probability measure on $\R$. We define the free convolution $\mu_1 \boxplus \mu_2$ as this unique probability measure, namely
\[
  G_{\mu_1 \boxplus \mu_2}(z) = G_{\mu_\ell}(\omega_\ell(z)), \qquad \forall \ell \in \{1, 2\}.
\]
\end{thm}

The tools of free harmonic analysis enable a great deal of practical computations. For example, if one takes $\mbf{A}_N$ to be a normalized matrix from the Gaussian unitary ensemble (GUE), then a classical result of Wigner shows that the LSD is the so-called semicircle distribution $\mu_{\mbf{A}} = \frac{1}{2\pi}\sqrt{4-t^2}\, dt$ \cite{Wig55}. At the same time, the unitary invariance of the GUE implies that $\mbf{A}_N$ is asymptotically free from a large class of random matrices \cite{Voi91}. In the setting above, one can take $\mbf{B}_N$ to be an independent diagonal matrix with i.i.d.\@ Rademacher entries, in which case $\mu_{\mbf{B}} = \frac{1}{2}\delta_{\pm 1}$. Implementing Theorem \ref{thm:subordination}, we obtain the LSD of the sum $\mbf{C}_N = \mbf{A}_N + \mbf{B}_N$:
\begin{gather*}
  \mu_{\mbf{C}} = \mu_{\mbf{A}} \boxplus \mu_{\mbf{B}} = \bigg(\frac{1}{2\pi}\sqrt{4-t^2}\, dt\bigg) \boxplus \bigg(\frac{1}{2}\delta_{\pm 1}\bigg) \\
  =\frac{1}{2\pi\sqrt{3}}\Bigg[\frac{\sqrt[3]{27t - 2t^3 + 3\sqrt{3}|t|\sqrt{27-4t^2}}}{\sqrt[3]{2}} - \frac{\sqrt[3]{2}t^2}{\sqrt[3]{27t - 2t^3 + 3\sqrt{3}|t|\sqrt{27-4t^2}}}\Bigg]\, dt.
\end{gather*}

Such additive perturbations appear naturally as models of interaction and noise. Under suitable conditions, we see that free probability allows us to understand the spectral distribution at the aggregate level; however, this approach fails to capture the behavior of the extremal eigenvalues. Indeed, consider the case of a rank one perturbation $\mbf{B}_N = \theta\mbf{E}_N^{(1, 1)}$, where $\mbf{E}_N^{(j, k)}$ is the matrix unit in the $(j, k)$-th coordinate and $\theta \in \R$. For $\mbf{A}_N$ GUE as before, the free convolution calculation $\mu_{\mbf{C}} = \mu_{\mbf{A}} \boxplus \mu_{\mbf{B}}$ reduces to the trivial identity
\[
\mu \boxplus \delta_0 = \mu, \qquad \forall \mu \in \mcal{P}(\R).
\]
From this perspective, the effect of the perturbation $\mbf{B}_N = \theta \mbf{E}_N^{(1, 1)}$ appears no different than the unperturbed model $\mbf{B}_N = 0$.

In actuality, we know that the behavior of the extremal eigenvalue exhibits a phase transition depending on the magnitude $|\theta|$ of the perturbation (a so-called BBP transition in view of the original work \cite{BBP05} on complex sample covariance matrices). In the case of the deformed GUE, P\'{e}ch\'{e} showed that the fluctuations of the extremal eigenvalue deviate from the Tracy-Widom distribution \cite{TW94} when $|\theta| \geq 1$ with the extremal eigenvalue even separating from the bulk when $|\theta| > 1$ \cite{Pec06}. The unitary invariance of the GUE implies that the same result holds for any rank one self-adjoint perturbation with nontrivial eigenvalue $\theta$. F\'{e}ral and P\'{e}ch\'{e} then extended the result to complex sub-Gaussian Wigner matrices under the perturbation $\mbf{B}_N' = \frac{\theta}{N}\mbf{J}_N$, where $\mbf{J}_N$ is the all-ones matrix \cite{FP07}. Notably, they proved the universality of the fluctuations of the extremal eigenvalue (cf. \cite{FK81,Sos99}). Ma\"{i}da established a large deviation principle for the extremal eigenvalue of the deformed Gaussian ensembles: as a corollary, this proves the same bulk separation phenomenon for the deformed Gaussian orthogonal ensemble (GOE) \cite{Mai07}. Capitaine, Donati-Martin, and F\'{e}ral generalized the bulk separation phenomenon to finite-rank perturbations: for example, $\mbf{B}_N$ of the form $\sum_{j = 1}^{N_0} \theta_j \mbf{E}_N^{(j, j)}$ for some fixed $N_0$. In this case, multiple eigenvalues exit the bulk, one for each value of $|\theta_j| > 1$. Their result holds for general Wigner matrices, real and complex, under the technical assumption that the entries satisfy a Poincar\'{e} inequality. At the same time, they extended the universality of the fluctuations of the extremal eigenvalue under perturbations of the form $\mbf{B}_N' = \frac{\theta}{N}\mbf{J}_N$ to real Wigner matrices. In contrast, they also proved the \emph{non-universality} of the fluctuations of the extremal eigenvalue under perturbations of the form $\mbf{B}_N = \theta \mbf{E}_N^{(1, 1)}$ \cite{CDMF09}. In a later work, the same authors also determined the joint fluctuations of the extremal eigenvalues \cite{CDMF12}. Pizzo, Renfrew, and Soshnikov \cite{PRS13} and later Renfrew and Soshnikov \cite{RS13} removed the technical assumptions for these results: the version we state below is due to them. For additional reading and related results, see the surveys \cite{Pec14,CDM17}.

\begin{thm}[BBP transition]\label{thm:bbp_transition}
For each $N \in \N$, let $(X_{j, k}^{(N)})_{j \leq k \in [N]}$ be a family of independent random variables, the off-diagonal entries $j < k$ possibly being complex-valued. We assume that the diagonal entries $j = k$ are centered with uniformly bounded variance satisfying the Lindeberg condition:
\begin{gather*}
  \E[X_{j, j}^{(N)}] = 0, \qquad \forall j \in [N]; \\
  \sup_{N \in \N} \sup_{j \in [N]} \E\big[|X_{j, j}^{(N)}|^2\big] < \infty; \\
  \lim_{N \to \infty} \frac{1}{N} \sum_{j \in [N]} \E\big[|X_{j, j}^{(N)}|^2\indc{|X_{j, j}^{(N)}| \geq \varepsilon \sqrt{N}}\big] = 0, \qquad \forall \varepsilon > 0.
\end{gather*}
For $(X_{j ,k}^{(N)})_{j < k \in [N]}$ real-valued, we assume that the off-diagonal entries $j < k$ are centered with identical variance and uniformly bounded fourth moments satisfying a Lindeberg type condition:
\begin{gather*}
  \E[X_{j, k}^{(N)}] = 0 \quad \text{and} \quad \E\big[|X_{j, k}^{(N)}|^2\big] = \sigma^2, \qquad \forall j < k \in [N]; \\
  \sup_{N \in \N} \sup_{j < k \in [N]} \E\big[|X_{j, k}^{(N)}|^4\big] < \infty; \\
  \lim_{N \to \infty} \frac{1}{N^2} \sum_{j < k \in [N]} \E\big[|X_{j, k}^{(N)}|^4\indc{|X_{j, k}^{(N)}| \geq \varepsilon N^{1/4}}\big] = 0, \qquad \forall \varepsilon > 0.  
\end{gather*}
For $(X_{j ,k}^{(N)})_{j < k \in [N]}$ complex-valued, we assume that the real and imaginary parts of each off-diagonal entry are independent with identical variance in addition to the conditions above. As a consequence,
\begin{gather*}
\E\big[|\Re(X_{j, k}^{(N)})|^2\big] = \E\big[|\Im(X_{j, k}^{(N)})|^2\big] = \frac{\sigma^2}{2}, \qquad \forall j < k \in [N].
\end{gather*}
Let $\mbf{X}_N(j, k) = X_{j, k}^{(N)}$ denote the corresponding (unnormalized) Wigner matrix with the usual normalization $\mbf{W}_N = \frac{1}{\sqrt{N}}\mbf{X}_N$. Assume that $\mbf{P}_N$ is a deterministic self-adjoint matrix of the same symmetry class as $\mbf{W}_N$ with fixed rank $r$ independent of the dimension. We further assume that the non-trivial eigenvalues of $\mbf{P}_N$ are independent of $N$, say $\theta_1 > \cdots > \theta_L$, where $\theta_\ell \neq 0$ occurs with multiplicity $m_\ell$ for $\ell \in [L]$. Let $L_{+\sigma} = \#(\ell \in [L]: \theta_\ell > \sigma)$ and $L_{-\sigma} = \#(\ell \in [L]: \theta_\ell < -\sigma)$, and define
\[
  \rho_\theta = \theta + \frac{\sigma^2}{\theta}.
\]
Then we have the following asymptotic behavior at the edge of the spectrum of the deformed Wigner ensemble $\mbf{W}_N + \mbf{P}_N$:
\begin{enumerate}[label=(\roman*)]
\item For any $\ell \in [L_{+\sigma}]$ and $i \in [m_\ell]$,
  \[
    \lambda_{m_1 + \cdots + m_{\ell-1} + i}(\mbf{W}_N + \mbf{P}_N) \pto \rho_{\theta_\ell};
  \]
\item $\lambda_{m_1 + \cdots + m_{L_{+\sigma}} + 1}(\mbf{W}_N + \mbf{P}_N) \pto 2\sigma$;
\item $\lambda_{N - m_L - \cdots - m_{L - L_{-\sigma} + 1}}(\mbf{W}_N + \mbf{P}_N) \pto -2\sigma$;
\item For any $\ell \in [L_{-\sigma}]$ and $i \in [m_{L - \ell + 1}]$,
  \[
    \lambda_{N - m_L - \cdots - m_{L - \ell + 1} + i}(\mbf{W}_N + \mbf{P}_N) \pto \rho_{\theta_{L-\ell+1}},
  \]
\end{enumerate}
where $\pto$ denotes convergence in probability.
\end{thm}

Recall that our earlier free convolution calculation failed to identify such outliers. Nevertheless, it turns out that the behavior of the outlying eigenvalues (as well as their eigenvectors) can be understood in terms of the subordination functions $\omega_\ell$ from Theorem \ref{thm:subordination} \cite{CDMFF11,Cap13,BBCF17} (see also \cite{BGN11} for related results). This suggests that free probability may yet prove useful to this end. Shlyakhtenko explained this connection using the framework of \emph{infinitesimal free probability}, an extension of free probability to the first order. In particular, by calculating a \emph{type B free convolution}, one obtains the $\frac{1}{N}$ correction to the LSD of such deformed ensembles. The outlying eigenvalues then appear in this correction in the form of Dirac masses \cite{Shl18}. We review this framework in the next section.

\subsection{Background}\label{sec:background}

We begin by recalling the usual free probability framework.
\begin{defn}[Free probability]\label{def:free_probability}
By a \emph{non-commutative (NC) probability space} $(\mcal{A}, \varphi)$, we mean a unital algebra $\mcal{A}$ over $\C$ paired with a unital linear functional $\varphi: \mcal{A} \to \C$. We say that $\varphi$ is \emph{tracial} if $\varphi(ab) = \varphi(ba)$ for all $a, b \in \mcal{A}$. The \emph{distribution} of a family of random variables $\mbf{a} = (a_i)_{i \in I} \subset \mcal{A}$ is the linear functional
\[
  \mu_{\mbf{a}}: \C\langle\mbf{x}\rangle \to \C, \qquad P \mapsto \varphi(P(\mbf{a})),
\]
where $\mbf{x} = (x_i)_{i \in I}$ is a set of non-commuting indeterminates and $P(\mbf{a}) \in \mcal{A}$ is the usual evaluation of NC polynomials. A sequence of families $(\mbf{a}_N)_{N \in \N}$, each living in a possibly different NC probability space $(\mcal{A}_N, \varphi_N)$, \emph{converges in distribution} if the sequence $(\mu_{\mbf{a}_N})_{N \in \N}$ converges pointwise. Note that the limit defines a new NC probability space $(\C\langle\mbf{x}\rangle, \lim_{N \to \infty} \mu_{\mbf{a}_N})$.

Unital subalgebras $(\mcal{A}_i)_{i \in I}$ of $\mcal{A}$ are said to be \emph{freely independent} (or simply \emph{free}) if for any $k \geq 2$ and consecutively distinct indices $i(1) \neq i(2) \neq \cdots \neq i(k)$,
\[
  \varphi(a_1 a_2 \cdots a_k) = 0, \qquad \forall a_j \in \interior{\mcal{A}}_{i(j)},
\]
where $\interior{\mcal{A}}_{i(j)} = \{a \in \mcal{A}_{i(j)}: \varphi(a) = 0\}$ denotes the subspace of centered elements. We say that collections of random variables $(\mcal{S}_i)_{i \in I}$ are free if the unital subalgebras that they generate are free. If a sequence of families $(\mbf{a}_N)_{N \in \N}$ converges in distribution, then we say that the random variables $\mbf{a}_N = (a_N^{(i)})_{i \in I}$ are \emph{asymptotically free} if the indeterminates $\mbf{x} = (x_i)_{i \in I}$ are free in $(\C\langle\mbf{x}\rangle, \lim_{N \to \infty} \mu_{\mbf{a}_N})$.
\end{defn}

\begin{rem}\label{rem:distribution}
The reader might wonder how the notion of a distribution above relates to the usual notion of a distribution for a real-valued random variable. If we assume both existence and uniqueness for the moment problem defined by $\mu_a$, then the two notions coincide. The moment sequences we consider in this paper will satisfy this assumption, so we speak of the two notions interchangeably. In particular, if $a, b \in (\mcal{A}, \varphi)$ are free with determinate moment problems, then $\mu_{a+b} = \mu_a \boxplus \mu_b$.
\end{rem}

\begin{eg}[Random matrices]\label{eg:random_matrices}
Let $\matN(L^{\infty-}(\Omega, \mcal{F}, \prob))$ denote the algebra of random $N \times N$ matrices whose entries, possibly complex-valued, have finite absolute moments of all orders. Then $(\matN(L^{\infty-}(\Omega, \mcal{F}, \prob)), \frac{1}{N}\E[{\Tr}(\cdot)])$ defines a tracial NC probability space.
\end{eg}

Voiculescu showed that independent unitarily invariant random matrices are asymptotically free \cite{Voi91}, the GUE being a prototypical example. Dykema later extended this result to general Wigner matrices \cite{Dyk93}. We now know freeness to be an ubiquitous phenomenon for invariant/mean-field multi-matrix models in the large $N$ limit \cite{MS17} (see also \cite{Spe17}).

Understanding the spectral behavior of non mean-field ensembles constitutes a major ongoing program of research, where random band matrices emerge as an attractive interpolative model (see \cite{Bou18} and the references therein). Here, the primary questions concern the local eigenvalue statistics and localization versus delocalization for the eigenvectors. In a different direction, we showed that freeness governs random band matrices for band widths $1 \ll b_N \ll N$ \cite{Au18}, motivating the investigations in this paper at the infinitesimal level. The results in \cite{Au18} rely on an extension of free probability introduced by Male called \emph{traffic probability} \cite{Mal11}: we make use of the traffic framework again, this time in conjunction with the infinitesimal framework. We refer the reader to \cite{Mal17,MP14,Gab15a,Gab15b,Gab15c,CDM16,ACDGM18} for additional reading on traffic probability and its applications.

Belinschi and Shlyakhtenko introduced infinitesimal free probability in \cite{BS12} to provide an analytic interpretation of the type $B$ free probability of Biane, Goodman, and Nica \cite{BGN03}. We content ourselves with the basic framework: for more on the interplay between these two notions, see \cite{FN10}. For recent work on infinitesimal free probability and its applications to random matrices, we mention the contributions \cite{Min18,DF19,Tse19}.

\begin{defn}[Infinitesimal free probability]\label{def:infinitesimal_free_probability}
By an \emph{infinitesimal NC probability space} $(\mcal{A}, \varphi, \varphi')$, we mean a NC probability space $(\mcal{A}, \varphi)$ with an additional linear functional $\varphi': \mcal{A} \to \C$ satisfying $\varphi'(1) = 0$. The \emph{infinitesimal distribution} of a family of random variables $\mbf{a} = (a_i)_{i \in I} \subset \mcal{A}$ is the linear functional
\[
  \nu_{\mbf{a}}: \C\langle\mbf{x}\rangle \to \C, \qquad P \mapsto \varphi'(P(\mbf{a})).
\]
We refer to the pair $(\mu_{\mbf{a}}, \nu_{\mbf{a}})$ as the \emph{type $B$ distribution} of $\mbf{a}$.

Unital subalgebras $(\mcal{A}_i)_{i \in I}$ of $\mcal{A}$ are said to be \emph{infinitesimally free} if
\begin{enumerate}[label=(\roman*)]
\item the $(\mcal{A}_i)_{i \in I}$ are free in $(\mcal{A}, \varphi)$;
\item for any $k \geq 2$ and consecutively distinct indices $i(1) \neq i(2) \neq \cdots \neq i(k)$, \label{def:Leibniz}
\[
  \varphi'(a_1 a_2 \cdots a_k) = \sum_{j = 1}^k \varphi(a_1 a_2 \cdots a_{j-1} \varphi'(a_j) a_{j+1} \cdots a_k), \qquad \forall a_j \in \interior{\mcal{A}}_{i(j)}.
\]
\end{enumerate}
Conditions (i) and (ii) are equivalent to the following asymptotic:
\[
  \varphi_t([a_1 - \varphi_t(a_1)] [a_2 - \varphi_t(a_2)] \cdots [a_k - \varphi_t(a_k)]) = O(t^2) \quad \text{as} \quad t \to 0,
\]
where $a_j \in \mcal{A}_{i(j)}$ and $\varphi_t = \varphi + t\varphi'$ for $t \in \R$. Thus, heuristically, we think of infinitesimal freeness as ``freeness to the first order''.
\end{defn}

\begin{rem}\label{rem:infinitesimal_distribution}
In view of Remark \ref{rem:distribution}, the reader might wonder how the notion of an infinitesimal distribution relates to the usual notion of a \emph{signed measure} on the real line. If we assume both existence and uniqueness for the signed moment problem defined by $\nu_a$, then the two notions coincide. The signed moment sequences we consider in this paper will typically satisfy this assumption, so we speak of the two notions interchangeably when possible. Note that the condition $\nu_a(1) = \varphi'(1) = 0$ implies that the corresponding signed measure has total mass zero.
\end{rem}

\begin{eg}[Random matrices, revisited]\label{eg:infinitesimal_random_matrices}
Let $\mcal{A}_N = (\mbf{A}_N^{(i)})_{i \in I}$ be a family of random matrices in $(\matN(L^{\infty-}(\Omega, \mcal{F}, \prob)), \frac{1}{N}\E[{\Tr}(\cdot)])$. Assume that $\mcal{A}_N$ converges in distribution with limit $\mu_{\mbf{x}} = \lim_{N \to \infty} \mu_{\mcal{A}_N}$. If we further assume that the limit
  \[
    \nu_{\mbf{x}} = \lim_{N \to \infty} N(\mu_{\mcal{A}_N} - \mu_{\mbf{x}})
  \]
exists, then $(\C\langle\mbf{x}\rangle, \mu_{\mbf{x}}, \nu_{\mbf{x}})$ defines a tracial infinitesimal NC probability space (both $\mu_{\mbf{x}}$ and $\nu_{\mbf{x}}$ vanish on the commutators). By a slight abuse of terminology, we often refer to $\nu_{\mbf{x}}$ (\emph{resp.,} $(\mu_{\mbf{x}}, \nu_{\mbf{x}})$) as the infinitesimal distribution (\emph{resp.,} type $B$ distribution) of $\mcal{A}_N$.
\end{eg}

In the single matrix case, say $\mbf{A}_N$, the infinitesimal distribution $\nu_{\mbf{A}}$ corresponds to the $\frac{1}{N}$ correction to the LSD $\mu_{\mbf{A}}$. Indeed, by definition,
\begin{align}
  \nu_{\mbf{A}}(x^\ell) &= \lim_{N \to \infty} N(\mu_{\mbf{A}_N}(x^\ell) - \mu_{\mbf{A}}(x^\ell)) \notag \\
                 &= \lim_{N \to \infty} \E[{\Tr}(\mbf{A}_N^\ell)] - N \Big(\lim_{M \to \infty} \mu_{\mbf{A}_M}(x^\ell)\Big) \notag \\
                 &= \lim_{N \to \infty} \E\Big[\sum_{k = 1}^N \lambda_k(\mbf{A}_N)^\ell\Big] - N \Big(\lim_{M \to \infty} \frac{1}{M}\E\Big[\sum_{j = 1}^M \lambda_j(\mbf{A}_M)^\ell\Big]\Big), \label{eq:unnormalized}
\end{align}
where we recall that $\mbf{A}_N$ is assumed to be self-adjoint. For example, in the case of $\mbf{A}_N \deq \op{GUE}(N, \frac{\sigma^2}{N})$, the infinitesimal distribution is null $\nu_{\mbf{A}} = 0$, a consequence of the genus expansion \cite{HZ86}. On the other hand, a result of Johansson \cite{Joh98} shows that the situation becomes much different for $\mbf{A}_N \deq \op{GOE}(N, \frac{\sigma^2}{N})$, where
\[
  \nu_{\mbf{A}} = \frac{1}{2}\bigg[\frac{1}{2}\delta_{\pm 2\sigma} - \frac{1}{\pi\sqrt{4\sigma^2-t^2}}\, dt \bigg].
\]
We mention that such corrections also exist for complex Wishart matrices \cite{MN04,Min18} and $\beta$-ensembles \cite{DE06}.

Note that the eigenvalues $(\lambda_k(\mbf{A}_N))_{k \in [N]}$ appear in \eqref{eq:unnormalized} via the \emph{unnormalized} trace. This suggests that the infinitesimal distribution is sensitive to outliers. To see this, we will need the following subordination result for the \emph{type $B$ free (additive) convolution}.

\begin{thm}[\hspace{-1sp}\cite{BS12}]\label{thm:type_B_subordination}
Suppose that $a, b \in (\mcal{A}, \varphi, \varphi')$ are infinitesimally free with compactly supported type $B$ distributions $(\mu_a, \nu_a), (\mu_b, \nu_b) \in \mcal{P}(\R) \times \mcal{M}_0(\R)$. By this, we mean that both coordinates of the type $B$ distribution have compact support. Then, in the notation of Theorem \ref{thm:subordination}, the sum $a + b$ also has a compactly supported type $B$ distribution $(\mu_{a + b}, \nu_{a + b}) \in \mcal{P}(\R) \times \mcal{M}_0(\R)$ characterized by 
\begin{enumerate}[label=(\roman*)]
\item $\mu_{a + b} = \mu_a \boxplus \mu_b$; 
\item $G_{\nu_{a+b}}(z) = G_{\nu_a}(\omega_a(z))\omega_a'(z) + G_{\nu_b}(\omega_b(z))\omega_b'(z)$,
\end{enumerate}
where $\omega_a'(z), \omega_b'(z)$ denote the usual derivatives. We define the type $B$ convolution $(\mu_a, \nu_a) \boxplus_B (\mu_b, \nu_b)$ as this unique type $B$ distribution, namely
\[
  (\mu_a, \nu_a) \boxplus_B (\mu_b, \nu_b) = (\mu_{a + b}, \nu_{a + b}).
\]
\end{thm}

\begin{thm}[\hspace{-1pt}\cite{Shl18}]\label{thm:infinitesimal_freeness_gaussian}
Let $\mbf{W}_N \deq \op{GUE}/\op{GOE}(N, \frac{\sigma^2}{N})$. Then for any fixed $N_0$, the matrices $\mbf{W}_N$ and $(\mbf{E}_N^{(j, k)})_{j, k \in [N_0]}$ are asymptotically infinitesimally free.
\end{thm}

Of course, we can easily compute the type $B$ distribution of the matrix units. For $\mbf{P}_N = \sum_{j = 1}^{N_0} \theta_j \mbf{E}_N^{(j, j)}$, we see that
\begin{align*}
  \lim_{N \to \infty} \mu_{\mbf{P}_N} &= \delta_0; \\
  \lim_{N \to \infty} N(\mu_{\mbf{P}_N} - \delta_0) &= \sum_{j = 1}^{N_0} \delta_{\theta_j} - N_0\delta_0.
\end{align*}
Using Theorem \ref{thm:type_B_subordination}, one obtains the $\frac{1}{N}$ correction to the LSD of the deformed Gaussian ensemble $\mbf{W}_N + \mbf{P}_N$ (cf. Theorem \ref{thm:bbp_transition}).
\begin{cor}[\hspace{-1pt}\cite{Shl18}]\label{cor:outliers_gaussian}
If $\mbf{W}_N \deq \op{GUE}(N, \frac{\sigma^2}{N})$ and $\mbf{P}_N = \sum_{j = 1}^{N_0} \theta_j \mbf{E}_N^{(j, j)}$, then the type $B$ distribution of $\mbf{W}_N + \mbf{P}_N$ is given by
\begin{align*}
  &\bigg(\frac{1}{2\pi\sigma^2} \sqrt{4\sigma^2 - t^2}\, dt, 0\bigg) \boxplus_B \bigg(\delta_0, \sum_{j = 1}^{N_0} \delta_{\theta_j} - N_0\delta_0\bigg) \\
  = &\bigg(\frac{1}{2\pi\sigma^2} \sqrt{4\sigma^2 - t^2}\, dt, \sum_{\substack{j \in [N_0]: \\ |\theta_j| \geq \sigma}} \delta_{\theta_j + \frac{\sigma^2}{\theta_j}} - \sum_{j \in [N_0]} \nu_j \bigg),
\end{align*}
where
\[
  \nu_j = \frac{\theta_j(t - 2\theta_j)}{2\pi(\theta_j(t - \theta_j) - \sigma^2)\sqrt{4\sigma^2 - t^2}}\, dt
\]
is a probability measure if $|\theta_j| \geq \sigma$; otherwise, $\nu_j$ is a signed measure of total mass zero with Jordan decomposition $\nu_j = \nu_j^+ - \nu_j^-$, where
\[
  \nu_j^+ =
  \begin{dcases}
    \indc{t \in [-2\sigma, 2\theta_j]}\frac{d\nu_j}{dt} & \text{if $\theta_j > 0$}; \\
    \indc{t \in [2\theta_j, 2\sigma]}\frac{d\nu_j}{dt}  & \text{if $\theta_j < 0$}.
  \end{dcases}
\]

If instead $\mbf{W}_N \deq \op{GOE}(N, \frac{\sigma^2}{N})$, then the type $B$ distribution of $\mbf{W}_N + \mbf{P}_N$ is given by
\begin{align*}
  &\bigg(\frac{1}{2\pi\sigma^2} \sqrt{4\sigma^2-t^2}\, dt,  \frac{1}{2}\bigg[\frac{1}{2}\delta_{\pm 2\sigma} - \frac{1}{\pi\sqrt{4\sigma^2-t^2}}\, dt\bigg]\bigg) \boxplus_B \bigg(\delta_0, \sum_{j = 1}^{N_0} \delta_{\theta_j} - N_0\delta_0\bigg) \\
  = &\bigg(\frac{1}{2\pi\sigma^2} \sqrt{4\sigma^2-t^2}\, dt,  \frac{1}{2}\bigg[\frac{1}{2}\delta_{\pm 2\sigma} - \frac{1}{\pi\sqrt{4\sigma^2-t^2}}\, dt\bigg] + \sum_{\substack{j \in [N_0]: \\ |\theta_j| \geq \sigma}} \delta_{\theta_j + \frac{\sigma^2}{\theta_j}} - \sum_{j \in [N_0]} \nu_j \bigg),
\end{align*}
where $\nu_j$ is as before.
\end{cor}

The proof of Theorem \ref{thm:infinitesimal_freeness_gaussian} relies on Wick's formula for Gaussian integration. Naturally, one can ask if the result extends to general Wigner matrices. In this case, one needs to first prove the existence of an infinitesimal distribution for the single matrix model, a calculation carried out by Enriquez and M\'{e}nard (see also \cite{KKP96}). We state a slight generalization of their result to allow for entries with possibly different distributions: the proof remains unchanged.

\begin{thm}[\hspace{-1pt}\cite{EM16}]\label{thm:infinitesimal_wigner}
For each $N \in \N$, let $(X_{j, k}^{(N)})_{j \leq k \in [N]}$ be a family of independent random variables, the off-diagonal entries $j < k$ possibly being complex-valued. We assume that the diagonal entries $j = k$ are centered with identical variance:
\[
  \E[X_{j, j}^{(N)}] = 0 \quad \text{and} \quad \E\big[|X_{j, j}^{(N)}|^2\big] = s^2, \qquad \forall j \in [N].
\]
For $(X_{j ,k}^{(N)})_{j < k \in [N]}$ real-valued ($\beta = 1$), we assume that the off-diagonal entries $j < k$ are centered with identical variance and fourth moments:
\[
  \E[X_{j, k}^{(N)}] = 0, \quad \E\big[|X_{j, k}^{(N)}|^2\big] = \sigma^2, \quad \text{and} \quad \E\big[|X_{j, k}^{(N)}|^4\big] = \alpha, \qquad \forall j < k \in [N].
\]
For $(X_{j ,k}^{(N)})_{j < k \in [N]}$ complex-valued ($\beta = 2$), we assume that the pseudo-variance of each off-diagonal entry vanishes in addition to the conditions above:
\[
  \E\big[(X_{j ,k}^{(N)})^2\big] = 0, \qquad \forall j < k \in [N].
\]
Lastly, we assume a strong uniform control on the moments:
\[
  \sup_{N \in \N} \sup_{j \leq k \in [N]} \E\big[|X_{j, k}^{(N)}|^\ell\big] < \infty, \qquad \forall \ell \in \N.
\]
Then the corresponding Wigner matrix $\mbf{W}_N(j, k) = \frac{1}{\sqrt{N}}X_{j, k}^{(N)}$ has an infinitesimal distribution $\nu = \frac{1}{2}\Big[\frac{\indc{\beta = 1}}{2} \delta_{\pm 2\sigma} + \nu_{\op{ac}}\Big]$, where\small
\[
  \nu_{\op{ac}} = \bigg[\Big(\frac{\alpha}{\sigma^4} + \beta - 4\Big)t^4 + \Big(\frac{s^2}{\sigma^2} -4\frac{\alpha}{\sigma^4} - 3\beta + 13\Big)t^2 + 2\Big(\frac{\alpha}{\sigma^4} - \frac{s^2}{\sigma^2} - 2\Big) + \beta\bigg] \frac{1}{\pi\sqrt{4\sigma^2 - t^2}}\, dt.
\]\normalsize
\end{thm}

\subsection{Statement of results}\label{sec:results}

Our first result extends Theorem \ref{thm:infinitesimal_freeness_gaussian} to general Wigner matrices. We also consider perturbations of the form $\frac{\theta}{N}\mbf{J}_N$, where we recall that $\mbf{J}_N$ is the all-ones matrix.

\begin{thm}\label{thm:infinitesimal_freeness_wigner}
Let $\mbf{W}_N$ be a Wigner matrix of the form in Theorem \ref{thm:infinitesimal_wigner}. Then for any fixed $N_0$, the matrices $\mbf{W}_N$, $(\mbf{E}_N^{(j, k)})_{j, k \in [N_0]}$, and $\frac{1}{N}\mbf{J}_N$ are asymptotically infinitesimally free.
\end{thm}

Note that the type $B$ distribution of $\frac{\theta}{N}\mbf{J}_N$ is identical to that of $\theta \mbf{E}_N^{(j, j)}$, allowing us to essentially repeat the calculation of Corollary \ref{cor:outliers_gaussian}.

\begin{cor}\label{cor:outliers_wigner}
The type $B$ distribution of the deformed Wigner ensemble
\[
  \mbf{W}_N + \sum_{j = 1}^{N_0} \theta_j \mbf{E}_N^{(j, j)} + \frac{\theta_{N_0 + 1}}{N}\mbf{J}_N
\]
is given by
\begin{gather*}
  \bigg(\frac{1}{2\pi\sigma^2} \sqrt{4\sigma^2 - t^2}\, dt, \nu\bigg) \boxplus_B \bigg(\delta_0, \sum_{j = 1}^{N_0} \delta_{\theta_j} - N_0\delta_0\bigg) \boxplus_B \bigg(\delta_0, \delta_{\theta_{N_0 +1}} - \delta_0\bigg)\\
  = \bigg(\frac{1}{2\pi\sigma^2} \sqrt{4\sigma^2 - t^2}\, dt, \nu + \sum_{\substack{j \in [N_0 + 1]: \\ |\theta_j| \geq \sigma}} \delta_{\theta_j + \frac{\sigma^2}{\theta_j}} - \sum_{j \in [N_0 + 1]} \nu_j \bigg),
\end{gather*}
where $\nu$ is as in Theorem \ref{thm:infinitesimal_wigner} and $\nu_j$ is as in Corollary \ref{cor:outliers_gaussian}.
\end{cor}

\begin{rem}\label{rem:perturbations}
The result above shows that while the infinitesimal distribution is sensitive to outliers, it fails to distinguish their fluctuations. Indeed, recall that the fluctuations of the extremal eigenvalue under perturbations of the form $\theta \mbf{E}_N^{(1, 1)}$ (\emph{resp.,} $\frac{\theta}{N}\mbf{J}_N$) are non-universal (\emph{resp.,} universal) for $|\theta| > \sigma$, whereas the infinitesimal distribution of $\mbf{W}_N + \theta\mbf{E}_N^{(1, 1)}$ and $\mbf{W}_N + \frac{\theta}{N}\mbf{J}_N$ are identical. In general, the fluctuations of the extremal eigenvalues depend on the geometry of the eigenvectors of the perturbation: localized (as in the case of $\sum_{j = 1}^{N_0} \theta_j \mbf{E}_N^{(j, j)}$) versus delocalized (as in the case of $\frac{\theta_{N_0 + 1}}{N}\mbf{J}_N$) \cite{CDMF12}.
\end{rem}

The usual strategy for studying outliers relies on a fine analysis of the resolvent, using delicate estimates currently unavailable for non mean-field ensembles. In contrast, the purview of the infinitesimal framework extends quite naturally to random band matrices. We restrict ourselves to the idealized situation of a periodically banded GUE matrix.

\begin{defn}[Random band matrix]\label{def:rbm}
Let $\mbf{X}_N \deq \op{GUE}(N, \sigma^2)$. For a band width $b_N \geq 0$, we define $\mbf{B}_N$ to be the corresponding periodic band matrix of ones:
\[
  \mbf{B}_N(j, k) = \indc{|j - k|_N \leq b_N},
\]
where
\begin{equation}\label{eq:periodic_band_width}
  |j - k|_N = \min(|j - k|, N - |j - k|).
\end{equation}
We assume that the band width $b_N \to \infty$, and we set
\[
  \xi_N = \min(2b_N + 1, N).
\]
We call the random matrix
\[
  \mbf{\Xi}_N = \frac{1}{\sqrt{\xi_N}} \mbf{B}_N \circ \mbf{X}_N
\]
a \emph{(normalized) periodically banded GUE matrix (of band width $b_N$)}. Of course, if $b_N \geq \floor{N/2}$, then $\mbf{\Xi}_N \deq \op{GUE}(N, \frac{\sigma^2}{N})$.
\end{defn}

Bogachev, Molchanov, and Pastur proved that the ESD $\mu(\mbf{\Xi}_N)$ converges weakly almost surely to the semicircle distribution \cite{BMP91}. In particular, this holds regardless of the rate $b_N \to \infty$ because of the periodic band width structure \eqref{eq:periodic_band_width}. We considered the multi-matrix case in \cite{Au18}, where it was shown that independent copies $(\mbf{\Xi}_N^{(i)})_{i \in I}$ of $\mbf{\Xi}_N$ are asymptotically free, regardless of the relative rates of growth of the band widths $(b_N^{(i)})_{i \in I}$. So, for example, it could be that
\[
  b_N^{(1)}, b_N^{(2)} \ll \sqrt{N} \ll b_N^{(3)}, b_N^{(4)}.
\]
We highlight this homogeneity around $\sqrt{N}$ because of its conjectural role, confirmed at the level of physical rigor, as the critical value for the localization-delocalization transition for random band matrices (again, see \cite{Bou18} for a recent survey).

While the rate $b_N \to \infty$ did not play a role in our calculations at the zeroth order, a $\sqrt{N}$ factor appears quite naturally at the first order. Our next result proves a sharp transition for the infinitesimal distribution around this rate.

\begin{thm}\label{thm:infinitesimal_band}
Let $\mbf{\Xi}_N$ be a periodically banded GUE matrix of band width $b_N$. Then for any $\ell \in \N$,
\[
  \lim_{N \to \infty} \E[{\Tr}(\mbf{\Xi}_N^{2\ell})] - N\sigma^{2\ell}{\Cat}(\ell) =
  \begin{dcases}
    0 &\text{if $b_N \gg \sqrt{N}$};\\
    \infty &\text{if $b_N \ll \sqrt{N}$};\\
    m_{2\ell}(\sigma^2, c) &\text{if $\lim_{N \to \infty} \frac{b_N}{\sqrt{N}} = c \in (0, \infty)$}. 
  \end{dcases}
\]
where ${\Cat}(\ell) = \frac{\binom{2\ell}{\ell}}{\ell+1}$ is the $\ell$th Catalan number, $m_2(\sigma^2, c) = 0$, and $m_{2\ell}(\sigma^2, c) \in (0, \infty)$ for $\ell \geq 2$. In particular, if $b_N \gg \sqrt{N}$, then the type $B$ distribution of $\mbf{\Xi}_N$ exists and agrees with that of a usual GUE matrix $\mbf{W}_N$. 
\end{thm}

The numbers $m_{2\ell}(\sigma^2, c)$ correspond to sums of volumes of regions cut out of a hypercube and satisfy
\begin{equation}\label{eq:bounded_support}
  1 \leq \liminf_{\ell \to \infty} [m_{2\ell}(\sigma^2, c)]^{\frac{1}{2\ell}} \leq \limsup_{\ell \to \infty} [m_{2\ell}(\sigma^2, c)]^{\frac{1}{2\ell}} \leq 2\sigma.
\end{equation}
Thus, a solution to the signed moment problem defined by the sequence
\[
  0, 0, m_2(\sigma^2, c), 0, m_4(\sigma^2, c), \ldots
\]
would necessarily be unique; however, we do not prove existence. Nevertheless, given a finite limit for the infinitesimal distribution, we can consider the question of finite-rank perturbations.

\begin{thm}\label{thm:infinitesimal_freeness_band}
Let $\mbf{\Xi}_N$ be a periodically banded GUE matrix of band width $b_N$ such that $b_N \gg \sqrt{N}$ or $\lim_{N \to \infty} \frac{b_N}{N} = c \in (0, \infty)$. Then for any fixed $N_0$, the matrices $\mbf{\Xi}_N$, $(\mbf{E}_N^{(j, k)})_{j, k \in [N_0]}$, and $\frac{1}{N}\mbf{J}_N$ are asymptotically infinitesimally free.
\end{thm}

For band widths $b_N \gg \sqrt{N}$, this allows us to repeat the calculation of Corollary \ref{cor:outliers_gaussian}. In particular, we find outliers at the classical positions from the deformed Wigner ensemble.

\begin{cor}\label{cor:outliers_band}
For $b_N \gg \sqrt{N}$, the type $B$ distribution of the deformed RBM
\[
  \mbf{\Xi}_N + \sum_{j = 1}^{N_0} \theta_j \mbf{E}_N^{(j, j)} + \frac{\theta_{N_0 + 1}}{N}\mbf{J}_N
\]
is given by
\begin{gather*}
  \bigg(\frac{1}{2\pi\sigma^2} \sqrt{4\sigma^2 - t^2}\, dt, 0\bigg) \boxplus_B \bigg(\delta_0, \sum_{j = 1}^{N_0} \delta_{\theta_j} - N_0\delta_0\bigg) \boxplus_B \bigg(\delta_0, \delta_{\theta_{N_0 +1}} - \delta_0\bigg)\\
  = \bigg(\frac{1}{2\pi\sigma^2} \sqrt{4\sigma^2 - t^2}\, dt, \sum_{\substack{j \in [N_0 + 1]: \\ |\theta_j| \geq \sigma}} \delta_{\theta_j + \frac{\sigma^2}{\theta_j}} - \sum_{j \in [N_0 + 1]} \nu_j \bigg),
\end{gather*}
where $\nu_j$ is as in Corollary \ref{cor:outliers_gaussian}.
\end{cor}

\begin{rem}\label{rem:critical_rate}
A solution to the signed moment problem at the rate $b_N \asymp \sqrt{N}$ would allow us to deduce the type $B$ distribution of the corresponding deformed model: one simply needs to add the hypothetical signed measure to the infinitesimal distribution in Corollary \ref{cor:outliers_band}.
\end{rem}

In this article, we consider the BBP transition for random band matrices exclusively within the infinitesimal framework. Naturally, one can ask if the usual form of these results hold, namely, convergence in probability of the extremal eigenvalues and convergence in distribution of the fluctuations. This will be the subject of future work. In the next section, we record the outcome of numerical simulations for various band widths. Notably, the data suggests that the position of the outliers and their fluctuations extend below the rate $b_N \asymp \sqrt{N}$.

\section{Numerical simulations}\label{sec:numerical_simulations}
We consider the fluctuations of the largest eigenvalue under both localized and delocalized perturbations of our model separately. In particular, we record the data
\begin{align*}
  F_{N, 1}(b_n) &= \frac{\sqrt{\xi_N}}{\sigma\sqrt{\frac{\theta^2-\sigma^2}{\theta^2}}}\Big[\lambda_1(\mbf{\Xi}_N + \theta\mbf{E}_N^{(1, 1)}) - \Big(\theta + \frac{\sigma^2}{\theta}\Big)\Big]; \\
  F_{N, 2}(b_N) &= \frac{\sqrt{N}}{\sigma\sqrt{\frac{\theta^2-\sigma^2}{\theta^2}}}\Big[\lambda_1\Big(\mbf{\Xi}_N + \frac{\theta}{N}\mbf{J}_N\Big) - \Big(\theta + \frac{\sigma^2}{\theta}\Big)\Big]
\end{align*}
for 5000 realizations of the matrix $\mbf{\Xi}_N$, where $\sigma^2 = 1$, $\theta = 2$, and $N = 7776$. The peculiar choice of dimension allows for the precise band widths $b_N = N^{3/5} = 216$ and $b_N = N^{2/5} = 36$. For reference, we also consider the band width $b_N = \lfloor N/2 \rfloor$, in which case $\mbf{\Xi}_N$ reduces to the usual GUE and $F_{N,1}(\lfloor N/2 \rfloor), F_{N, 2}(\lfloor N/2 \rfloor) \dto \mcal{N}(0, 1)$ by a result of P\'{e}ch\'{e} \cite{Pec06}. We emphasize the difference in scaling between $F_{N, 1}(b_N)$ and $F_{N, 2}(b_N)$. Indeed, the data strongly suggests that we still have the convergence $F_{N,1}(b_N), F_{N, 2}(b_N) \dto \mcal{N}(0, 1)$ under the respective normalizations (even at the rate $b_N = N^{2/5} \ll \sqrt{N}$). The simulations were performed in Julia \cite{BEKS17} and the data plotted using Gadfly \cite{Jan18}.

\begin{figure}
\centering
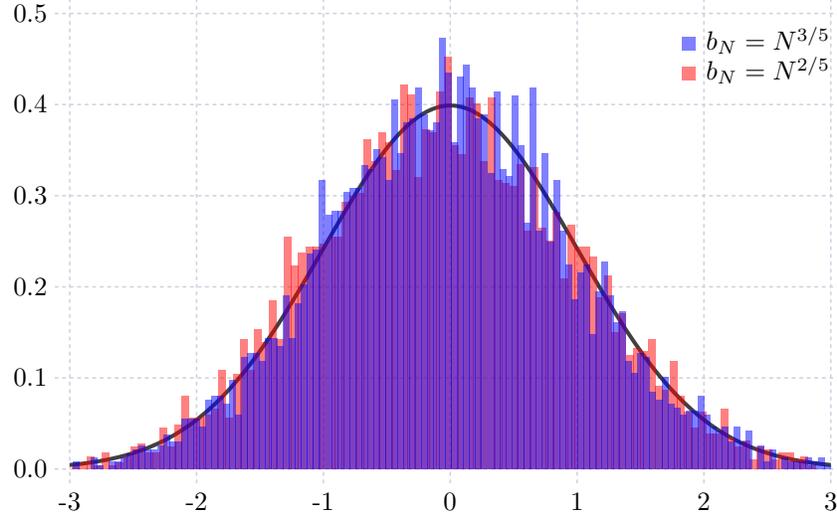
\caption{Histograms of the normalized eigenvalues $F_{N, 1}(b_N)$ overlaid for two values of $b_N$ and plotted against the conjectural limiting standard normal density. The overlapping region takes on the color blue $+$ red = purple.}
\label{fig:hist_localized}
\end{figure}

\begin{figure}
\centering
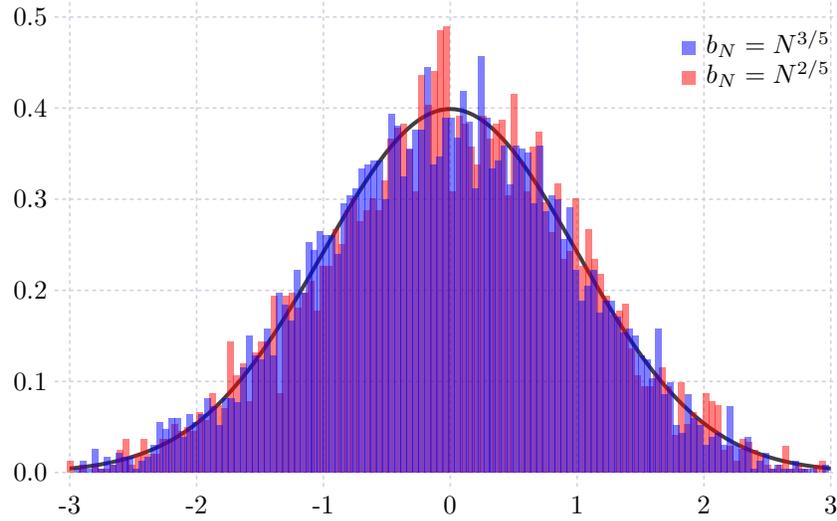  
\caption{Histograms of the normalized eigenvalues $F_{N, 2}(b_N)$ overlaid for two values of $b_N$ and plotted against the conjectural limiting standard normal density.}
\label{fig:hist_delocalized}
\end{figure}

\begin{figure}
\centering
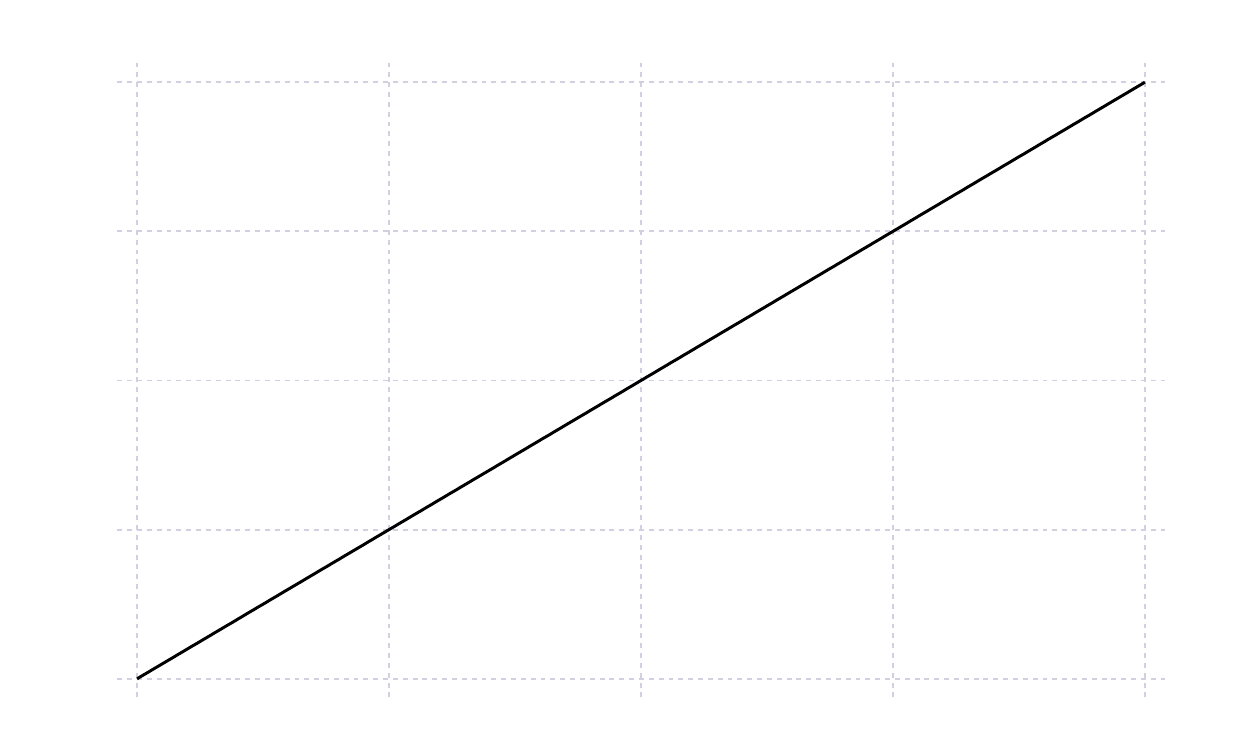
\caption{Quantile-quantile plots of the normalized eigenvalues $F_{N,1}(b_N)$ against the baseline data $F_{N, 1}(\lfloor N/2 \rfloor)$ from the deformed GUE overlaid for two values of $b_N$.}
\label{fig:qq_localized}
\end{figure}

\begin{figure}
\centering
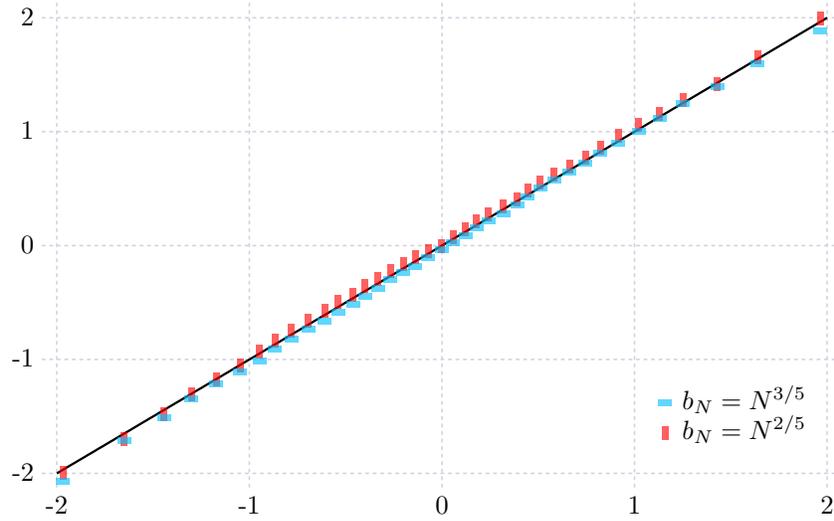
\caption{Quantile-quantile plots of the normalized eigenvalues $F_{N,2}(b_N)$ against the baseline data $F_{N, 2}(\lfloor N/2 \rfloor)$ from the deformed GUE overlaid for two values of $b_N$.}
\label{fig:qq_delocalized}
\end{figure}

The scaling in $F_{N, 1}$ should come as no surprise. To see this, note that the periodic band width structure in some sense reduces the trace expansion at each entry locally to that of a $\xi_N \times \xi_N$ matrix. So, heuristically, we think of $\theta\mbf{E}_N^{(1, 1)}$ as a perturbation of $\mbf{W}_{\xi_N} \deq \op{GUE}(\xi_N, \frac{\sigma^2}{\xi_N})$. On the other hand, in the case of $F_{N, 2}$, adding $\frac{\theta}{N}\mbf{J}_N$ forces us to consider the entire $N \times N$ matrix, removing any notion of homogeneity. Moreover, the entries of this perturbation come in at a different scale than our matrix entries $\mbf{\Xi}_N = \frac{1}{\sqrt{\xi_N}} \mbf{B}_N \circ \mbf{X}_N$. We can still make (non-rigorous) sense of the scaling in $F_{N, 2}$ by considering the trace expansion as a choice at each entry between the original matrix $\mbf{X}_N(i, j)$ and the perturbation $\frac{\theta}{N}$, where the first option is available iff $|i - j|_N \leq b_N$. But this precisely balances with the normalization of the entries in $\mbf{\Xi}_N$, and so the scaling should follow the usual case of $\mbf{W}_N + \frac{\theta}{N}\mbf{J}_N$.

For (undeformed) random band matrices, Sodin proved that the extremal eigenvalues converge to the edge of the support for band widths $b_N \gg \log(N)$ with the fluctuations exhibiting a crossover at the rate $b_N \asymp N^{5/6}$ \cite{Sod10}. The simulations do not support the idea of a similar crossover for the deformed model, suggesting that the perturbations regularize the fluctuations of the extremal eigenvalues.

\section{The infinitesimal distribution of a random band matrix}\label{sec:infinitesimal_distribution_rbm}

For convenience, we fix the variance $\sigma^2 = 1$ in this section: the general result follows from a simple scaling. Section \ref{sec:band_variant_genus} proves the existence of an infinitesimal distribution for a periodically banded GUE matrix in the regime $b_N = \Omega(\sqrt{N})$ using a band variant of the genus expansion. Section \ref{sec:finite_rank_perturbations} then proves the asymptotic infinitesimal freeness of our model from the matrix units and the normalized all-ones matrix, allowing us to carry out the advertised type $B$ free convolution calculation.

\subsection{A band variant of the genus expansion}\label{sec:band_variant_genus}

We consider traces in powers of our matrix $\mbf{\Xi}_N$. To begin, note that
\[
  \E[{\Tr}(\mbf{\Xi}_N^{2\ell-1})] = 0, \qquad \forall \ell \in \N.
\]
This follows from the usual symmetry argument, which still holds even in the presence of the band width condition. We turn our attention to the even powers, where we must now account for the band width explicitly:
\begin{align*}
  \E[{\Tr}(\mbf{\Xi}_N^{2\ell})] &= \sum_{\eta: [2\ell] \to [N]} \E\Big[\prod_{j = 1}^{2\ell} \mbf{\Xi}_N(\eta(j), \eta(j+1))\Big] \\
                               &= \xi_N^{-\ell} \sum_{\substack{\eta: [2\ell] \to [N] \text{ s.t.} \\ |\eta(j) - \eta(j+1)|_N \leq b_N}} \E\Big[\prod_{j = 1}^{2\ell} \mbf{X}_N(\eta(j), \eta(j+1))\Big],
\end{align*}
where $\eta(2\ell + 1) = \eta(1)$. Using Wick's formula, we obtain the expansion
\[
  \E[{\Tr}(\mbf{\Xi}_N^{2\ell})] = \xi_N^{-\ell} \sum_{\substack{\eta: [2\ell] \to [N] \text{ s.t.} \\ |\eta(j) - \eta(j+1)|_N \leq b_N}} \sum_{\pi \in \mcal{P}_2(2\ell)} \indc{\eta \circ \gamma \circ \pi = \eta},
\]
where $\gamma = (1, 2, \ldots, 2\ell) \in \mfk{S}_{2\ell}$. Here, we consider a pair partition $\pi$ as a $2\ell$-permutation when computing the composition $\eta \circ \gamma \circ \pi$. Interchanging the sums, we arrive at the expression
\begin{equation}\label{eq:genus_expansion_permutations}
  \E[{\Tr}(\mbf{\Xi}_N^{2\ell})] = \xi_N^{-\ell} \sum_{\pi \in \mcal{P}_2(2\ell)} Q(\ell, N, b_N, \pi),
\end{equation}
where
\begin{equation}\label{eq:admissible_labels}
  Q(\ell, N, b_N, \pi) = \#\mleft(\eta: [2\ell] \to [N] \mathrel{}\mathclose{}\middle|\mathopen{}\mathrel{} \begin{aligned} |\eta(j) - \eta(j+1)|_N \leq b_N \\
    \eta \text{ is constant on the cycles of $\gamma \circ \pi$}
\end{aligned}\mright).
\end{equation}
Note that we have the simple upper bound
\[
  Q(\ell, N, b_N, \pi) \leq N\xi_N^{\#(\gamma \circ \pi) - 1},
\]
where $\#(\gamma \circ \pi)$ denotes the number of cycles of $\gamma \circ \pi \in \mfk{S}_{2\ell}$. Indeed, starting with an arbitrary cycle of $\gamma \circ \pi$, say the cycle that contains 1, we have $N$ choices for the common index $\eta(1) \in [N]$ of the elements in this cycle. After making this choice, we must then choose the indices of the remaining cycles to satisfy the band width condition $|\eta(j) - \eta(j+1)|_N \leq b_N$, for which there are at most $\xi_N$ choices at each step. In general, this upper bound is strict: by the time you arrive to choose the index of a cycle of $\gamma \circ \pi$, you might have fewer than $\xi_N$ choices if the cycle is neighboring two cycles whose indices have already been chosen. As an example, take $\ell = 4$ and $\pi = (1, 5)(2, 8)(3, 7)(4, 6)$. In this case, $\gamma \circ \pi = (1, 6, 5, 2)(3, 8)(4, 7)$. Suppose that we pick the indices $\eta(1) = 1$ and $\eta(3) = 1+b_N$ for the cycles $(1, 6, 5, 2)$ and $(3, 10)$ respectively. Then the index $\eta(4)$ of the cycle $(4, 9)$ must satisfy both
\[
  |\eta(1) - \eta(4)|_N \leq b_N \quad \text{and} \quad |\eta(3) - \eta(4)|_N \leq b_N.
\]
If we assume that $b_N \ll N$, then we only have $1 + b_N < \xi_N$ choices for $\eta(4)$.

We quickly see the problem. By using up all of our leeway, we could potentially leave the indices too far apart to meet up again. For a simple parallel, consider placing three points $p_1, p_2, p_3$ in $\R^2$ such that any pair of points must be within unit distance of each other. Choosing the first point arbitrarily, say at the origin, and placing $p_2$ at $(1, 0)$, we can no longer place $p_3$ at an arbitrary point in the unit circle. This analogy gives us a lower bound for our original problem. If we instead divide our leeway by $\#(\gamma \circ \pi) - 1$ and pick the indices of the successive cycles arbitrarily at periodic distance less than or equal to this quotient, then we will stay within the permitted region (essentially just the triangle inequality). Thus, 
\begin{equation}\label{eq:bound_on_labels}
  N\bigg(\frac{\xi_N}{\#(\gamma \circ \pi) - 1}\bigg)^{\#(\gamma \circ \pi) - 1} \leq Q(\ell, N, b_N, \pi) \leq N\xi_N^{\#(\gamma \circ \pi) - 1}.
\end{equation}

We define a graph to keep track of the constraints on $\eta$ induced by cycles of $\gamma \circ \pi$ with adjacent elements $j, j+1$. Let $C_{2\ell}$ be the directed cycle graph on the vertices $V_{2\ell} = (v_j)_{j=1}^{2\ell}$ with edges $E_{2\ell} = (e_j)_{j=1}^{2\ell}$ in the direction $v_j \xrightarrow{e_j} v_{j+1}$. We equate the map $\eta: [2\ell] \to [N]$ with a labeling $\eta: V_{2\ell} \to [N]$ of the vertices in the obvious way. The edges then indicate the band width constraint by virtue of the equivalence
\[
  v_j \sim v_{j+1} \iff |\eta(v_j) - \eta(v_{j+1})|_N \leq b_N.
\]
At the moment, the direction of the edges do not play a role.

For a pair partition $\pi \in \mcal{P}_2(2\ell)$, we define $C_{2\ell}^\pi$ as the directed multigraph obtained from $C_{2\ell}$ by identifying the vertices $V_{2\ell}$ according to the blocks of $\pi$ as follows: if $(j < k)$ is a block of $\pi$, then we identify the source of the edge $e_j$ with the target of the edge $e_k$ (so $v_j \overset{\pi}{\sim} v_{k+1}$) and the source of the edge $e_k$ with the target of the edge $e_j$ (so $v_k \overset{\pi}{\sim} v_{j+1}$). In other words, for each block $(j < k) \in \pi$, we overlay the edges $e_j$ and $e_k$ head-to-tail. The vertices in the graph $C_{2\ell}^\pi$ correspond to the cycles of $\gamma \circ \pi$ with the edges indicating a constraint on the labels of the cycles induced by the constraint on the labels of the vertices. Note that the graph $C_{2\ell}^\pi$ might have loops. Of course, the constraint from a loop is vacuous, nor do the multiplicity/direction of the edges indicate any additional constraint at the level of the cycles of $\gamma \circ \pi$. So, we define $\underline{C}_{2\ell}^\pi$ as the underlying simple graph.

At the same time, we know that
\[
  \ell + 1 - 2\floor{\ell/2} \leq \#(\gamma \circ \pi) \leq \ell+1,
\]
where
\[
  \#(\gamma \circ \pi) = \ell + 1 \iff \pi \in \mcal{NC}_2(2\ell) \subset \mcal{P}_2(2\ell)
\]
by a result of Biane \cite{Bia97}. In particular, if $\pi \in \mcal{NC}_2(2\ell)$, then the graph $C_{2\ell}^\pi$ is a double tree in the sense of Male \cite{Mal11}. By this, we mean that $C_{2\ell}^\pi$ has no loops and $\underline{C}_{2\ell}^\pi$ is a tree such that the multiplicity of each edge in $C_{2\ell}^\pi$ is two (so-called \emph{twin edges}). To see this, note that the graph $\underline{C}_{2\ell}^\pi = (\underline{V}_{2\ell}^\pi, \underline{E}_{2\ell}^\pi)$ is connected with $\#(\underline{V}_{2\ell}^\pi) = \#(\gamma \circ \pi) = \ell + 1$, which implies that $\#(\underline{E}_{2\ell}^\pi) \geq \ell$. At the same time, $C_{2\ell}^\pi$ is obtained from $C_{2\ell}$ by overlaying pairs of edges, whence $\#(\underline{E}_{2\ell}^\pi) \leq \ell$. Thus,
\[
  \#(\underline{E}_{2\ell}^\pi) = \ell = \#(\underline{V}_{2\ell}^\pi) - 1, \qquad \forall \pi \in \mcal{NC}_2(2\ell),
\]
as was to be shown. In the case of a tree $\underline{C}_{2\ell}^\pi$, we do not run into a problem when choosing the vertices greedily using the entire leeway at each step, and so the upper bound for $Q(\ell, N, b_N, \pi)$ in \eqref{eq:bound_on_labels} becomes an equality. In other words,
\begin{equation}\label{eq:double_tree}
  \pi \in \mcal{NC}_2(2\ell) \implies Q(\ell, N, b_N, \pi) = N\xi_N^{\#(\gamma \circ \pi) -1} = N\xi_N^\ell.
\end{equation}
Indeed, recall that our earlier counterexample
\[
  \pi = (1, 5)(2, 8)(3, 7)(4, 6) \not\in \mcal{NC}_2(8).
\]

\begin{figure}
\centering
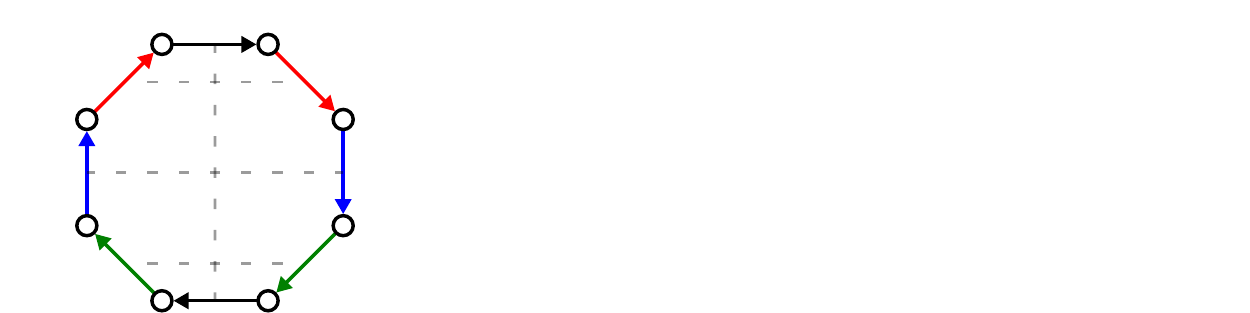
\caption{An example of the construction of $\underline{C}_{2\ell}^\pi$ for $\pi = (1, 5)(2, 8)(3, 7)(4, 6)$.}
\label{fig:genus_graph}
\end{figure}

Let $\mcal{C}_2(2\ell) = \mcal{P}_2(2\ell)\setminus \mcal{NC}_2(2\ell)$. Applying \eqref{eq:double_tree} to our earlier \eqref{eq:genus_expansion_permutations} and rearranging, we obtain
\[
  \E[{\Tr}(\mbf{\Xi}_N^{2\ell})] - N\cdot{\Cat}(\ell) = \xi_N^{-\ell} \sum_{\pi \in \mcal{C}_2(2\ell)} Q(\ell, N, b_N, \pi).
\]
Using our bounds \eqref{eq:bound_on_labels} for $Q(\ell, N, b_N, \pi)$, we see that
\begin{align*}
  \frac{N}{\xi_N^{\ell}} \sum_{\pi \in \mcal{C}_2(2\ell)} \bigg(\frac{\xi_N}{\#(\gamma \circ \pi) - 1}\bigg)^{\#(\gamma \circ \pi) - 1} &\leq \xi_N^{-\ell} \sum_{\pi \in \mcal{C}_2(2\ell)} Q(\ell, N, b_N, \pi) \\
                                                                                                                   &\leq N \sum_{\pi \in \mcal{C}_2(2\ell)} \xi_N^{\#(\gamma \circ \pi) - \ell - 1}.
\end{align*}
Since $\#(\gamma \circ \pi) \leq \ell + 1$, we also have the lower bound
\[
  \frac{N}{\ell^\ell} \sum_{\pi \in \mcal{C}_2(2\ell)} \xi_N^{\#(\gamma \circ \pi) - \ell - 1} \leq \frac{N}{\xi_N^{\ell}} \sum_{\pi \in \mcal{C}_2(2\ell)} \bigg(\frac{\xi_N}{\#(\gamma \circ \pi) - 1}\bigg)^{\#(\gamma \circ \pi) - 1}.
\]
At this point, we use the genus expansion to count the number of pair partitions $\pi$ that contribute to a given exponent $\#(\gamma \circ \pi) - \ell - 1$ appearing in the summands of our bounds. Altogether, this allows us to write
\begin{equation}\label{eq:band_genus}
  \frac{N}{\ell^\ell} \sum_{g = 1}^{\floor{\ell/2}} \frac{\varepsilon_g(\ell)}{\xi_N^{2g}} \leq \E[{\Tr}(\mbf{\Xi}_N^{2\ell})] - N\cdot{\Cat}(\ell) \leq N \sum_{g = 1}^{\floor{\ell/2}} \frac{\varepsilon_g(\ell)}{\xi_N^{2g}},
\end{equation}
where $\varepsilon_g(\ell) = \#(\pi \in \mcal{P}_{2\ell} : \#(\gamma \circ \pi) - \ell - 1 = -2g)$. Naturally, this calculation recovers the semicircle law. To see this, we simply take the normalized limit
\[
  \lim_{N \to \infty} \frac{1}{N} \bigg(\E[{\Tr}(\mbf{\Xi}_N^{2\ell})] - N\cdot{\Cat}(\ell) \bigg) = \lim_{N \to \infty} \sum_{g = 1}^{\floor{\ell/2}} \frac{\varepsilon_g(\ell)}{\xi_N^{2g}} = 0;
\]
however, without this normalization, the ratio $\frac{N}{\xi_N^2}$ arises in \eqref{eq:band_genus} as the leading order term. In particular, if $b_N \gg \sqrt{N}$, then $\frac{N}{\xi_N^2} = o(1)$ and
\[
  \lim_{N \to \infty} \E[{\Tr}(\mbf{\Xi}_N^{2\ell})] - N\cdot{\Cat}(\ell) = 0.
\]
In this case, the infinitesimal distribution of a periodically banded GUE matrix is null, which matches the calculation for the usual GUE. On the other hand, if $b_N \ll \sqrt{N}$, then the lower bound in \eqref{eq:band_genus} implies that
\[
  \lim_{N \to \infty} \E[{\Tr}(\mbf{\Xi}_N^{2\ell})] - N\cdot{\Cat}(\ell) = \infty.
\]
Finally, in the intermediate regime $\lim_{N \to \infty} \frac{b_N}{\sqrt{N}} = c \in (0, \infty)$, we see that 
\begin{equation}\label{eq:limits}
\begin{aligned}
  \frac{\varepsilon_1(\ell)}{4c^2\ell^\ell} &\leq \liminf_{N \to \infty} \E[{\Tr}(\mbf{\Xi}_N^{2\ell})] - N\cdot{\Cat}(\ell) \\
  &\leq \limsup_{N \to \infty} \E[{\Tr}(\mbf{\Xi}_N^{2\ell})] - N\cdot{\Cat}(\ell) \leq \frac{\varepsilon_1(\ell)}{4c^2},
\end{aligned}
\end{equation}               
where $\varepsilon_1(0) = \varepsilon_1(1) = 0$ and
\begin{equation}\label{eq:genus_one_asymptotic}
  \varepsilon_1(\ell) = \frac{(2\ell-1)!}{6(\ell-2)!(\ell-1)!} \sim \frac{1}{2}\sqrt{\frac{\ell}{\pi}}(\ell-1)4^\ell,
\end{equation}
the equality (\emph{resp.,} asymptotic) following from the three-term recurrence of Harer and Zagier \cite{HZ86} (\emph{resp.,} Stirling's formula).

Of course, one expects to be able to say more than just \eqref{eq:limits}, namely, that a limit exists in the intermediate regime (and hopefully with some nice formula or interpretation). This amounts to calculating
\[
  \lim_{N \to \infty} \frac{Q(\ell, N, b_N, \pi)}{\xi_N^\ell}
\]
for $\pi \in \mcal{P}_2(2\ell)$ such that $g(\pi) = 1$; however, the value of this limit crucially depends on the particular geometry of $\pi$. For example, consider the pair partitions
\begin{align*}
  \pi_1 &= (1, 3)(2, 4)(5, 6)(7, 8)(9, 10);\\
  \pi_2 &= (1, 6)(2, 10)(3, 9)(4, 8)(5, 7).
\end{align*}
Then $\#(\gamma \circ \pi_1) = \#(\gamma \circ \pi_2) = 4$, and so $g(\pi_1) = g(\pi_2) = 1$. Going through the graph construction, we see that $\underline{C}_{10}^{\pi_1}$ is a tree, which means that $Q(5, N, b_N, \pi_1) = N\xi_N^3$ attains the upper bound. On the other hand, $\underline{C}_{10}^{\pi_2}$ is the undirected cycle graph $\underline{C}_4$, which means that $Q(5, N, b_N, \pi_2) < tN\xi_N^3$ for some $t \in (0, 1)$.

In fact, the graph $\underline{C}_{2\ell}^\pi = (\underline{V}_{2\ell}^\pi, \underline{E}_{2\ell}^\pi)$ contains precisely the information we need to compute the limit. To see this, note that we can rewrite \eqref{eq:admissible_labels} as
\[
  Q(\ell, N, b_N, \pi) = \#(\eta: \underline{V}_{2\ell}^\pi \to [N] : |\eta(v) - \eta(w)|_N \leq b_N \text{ for every } \{v, w\} \in \underline{E}_{2\ell}^\pi).
\]
Indeed, the vertices of $\underline{C}_{2\ell}^\pi$ correspond to the cycles of $\gamma \circ \pi$ by construction and the edges function precisely to keep track of the band width constraint. This suggests computing the limit of $Q(\ell, N, b_N, \pi)$ as an integral over the hypercube $[0, 1]^{\#(V)}$ after scaling by $N^{\#(V)}$ (for example, as in \cite{Au18}); however, in this case, the band width $b_N \asymp \sqrt{N} \ll N$ and $\#(V) = \ell - 1$ (recall that $g(\pi) = 1$). So, the integral interpretation must take into account the vanishing scale of the mesh size $b_N$ and the difference in the scaling exponent $\#(V) \neq \ell$. To accomplish this, we use the fact that the periodic band width structure implies a certain homogeneity in our choice of admissible maps $\eta$. In particular, fixing a vertex $v_0 \in \underline{V}_{2\ell}^\pi$, we see that\small
\begin{equation}\label{eq:admissible_maps_fixed}
  \frac{Q(\ell, N, b_N, \pi)}{N} = \#\mleft(\eta: \underline{V}_{2\ell}^\pi \to [N] \mathrel{}\mathclose{}\middle|\mathopen{}\mathrel{} \begin{aligned} \eta(v_0) = 1 \\
    |\eta(v) - \eta(w)|_N \leq b_N, \quad \forall \{v, w\} \in \underline{E}_{2\ell}^\pi\end{aligned}\mright).
\end{equation}\normalsize
So, we consider the equivalent expression
\[
  \frac{Q(\ell, N, b_N, \pi)}{\xi_N^\ell} = \frac{N}{\xi_N^2}\frac{Q(\ell, N, b_N, \pi)}{N\xi_N^{\ell-2}},
\]
where
\[
  \lim_{N \to \infty} \frac{N}{\xi_N^2} = \lim_{N \to \infty} \frac{N}{(2b_N+1)^2} = \frac{1}{4c^2}. 
\]
This reduces the problem to computing
\[
  \lim_{N \to \infty} \frac{Q(\ell, N, b_N, \pi)}{N\xi_N^{\ell-2}} = \frac{1}{2^{\ell-2}}\lim_{N \to \infty} \frac{Q(\ell, N, b_N, \pi)}{Nb_N^{\ell-2}},
\]
which we can interpret using \eqref{eq:admissible_maps_fixed}. After fixing the label $\eta(v_0) = 1$, we must choose the labels of the remaining $\ell - 2$ vertices according to the band width constraint. The exponent of the normalization $b_N^{\ell-2}$ now matches the remaining degrees of freedom, and the base matches the maximum step size $|\eta(v) - \eta(w)|_N \leq b_N$.

The remaining issue concerns the region of integration. After the normalization, the step size $|\eta(v) - \eta(w)|_N \leq b_N$ becomes a single unit length. To ensure that the integral captures the full range of possibilities, we must choose a hypercube of appropriate side length. If $\eta(v_0) = 1$, then the image $\eta(V)$ will necessarily be disjoint from $[1+(\ell-2)b_N, N-(\ell-2)b_N]$. So, the hypercube $[0, 2(\ell-2)]^{\ell-2}$ will ensure that the region of integration is sufficiently large.

We can now give the integral representation for our limit. For $\ell \geq 2$, we define an integral $I_\ell^\pi$ associated to the graph $\underline{C}_{2\ell}^\pi = (\underline{V}_{2\ell}^\pi, \underline{E}_{2\ell}^\pi)$ as follows. Pick an arbitrary vertex $v_0 \in \underline{V}_{2\ell}^\pi$, and let $E_0 \subset \underline{E}_{2\ell}^\pi$ be the set of edges adjacent to $v_0$. We write $E_1 = \underline{E}_{2\ell}^\pi\setminus E_0$ for the remaining edges. By construction, the integral \small
\[
  I_\ell^\pi = \int_{[0, 2(\ell-2)]^{\ell-2}}  \prod_{\{v, v_0\} \in E_0} \indc{|t_v|_{2(\ell-2)} \leq 1} \prod_{\{v, w\} \in E_1} \indc{|t_v - t_w|_{2(\ell-2)} \leq 1} \prod_{v \in V \setminus \{v_0\}} dt_v
\]\normalsize
then satisfies
\[
  \lim_{N \to \infty} \frac{Q(\ell, N, b_N, \pi)}{\xi^{\ell}} = \frac{I_\ell^\pi}{2^\ell c^2},
\]
where
\[
  |\cdot|_{2(\ell-2)} = \min(|\cdot|, 2(\ell-2) - |\cdot|).
\]
For $\ell = 2$, we set $[0,0]^0 = \{0\}$ by convention, in which case the integral reduces to $I_2^\pi = 1$ for the only genus one partition $\pi = (1,3)(2,4) \in \mcal{P}_2(4)$. As an example, consider the case of $\ell = 4$, $\pi = (1, 5)(2, 8)(3, 7)(4, 6)$, and $b_N = \sqrt{N}$. Then
\[
  \frac{I_\ell^\pi}{2^\ell c^2} = \frac{1}{16}\int_{[0, 4]^2} \indc{|t|_4 \leq 1}\indc{|s|_4 \leq 1}\indc{|t-s|_4 \leq 1} \ dt\, ds= \frac{3}{16}.   
\]
One can easily verify that this agrees with the direct calculation
\[
  \lim_{N \to \infty} \frac{Q(\ell, N, b_N, \pi)}{\xi^\ell} = \lim_{N \to \infty} \frac{N(2b_N + 1 + 2\sum_{j = b_N + 1}^{2b_N} j)}{(2b_N+1)^4} = \frac{3}{16}.
\]
Thus, for $\lim_{N \to \infty} \frac{b_N}{\sqrt{N}} = c \in (0, \infty)$, we see that
\[
  m_{2\ell}(1, c) = \lim_{N \to \infty} \E[{\Tr}(\mbf{\Xi}_N^{2\ell})] - N\cdot{\Cat}(\ell) = \frac{1}{2^\ell c^2}\sum_{\substack{\pi \in \mcal{P}_2(2\ell):\\ g(\pi) = 1}} I_\pi^\ell.
\]
Altogether, this proves Theorem \ref{thm:infinitesimal_band}.

We use our earlier bound \eqref{eq:limits} and the asymptotic \eqref{eq:genus_one_asymptotic} for $\varepsilon_1(\ell)$ to see that
\[
  \limsup_{\ell \to \infty} [m_{2\ell}(1,c)]^{\frac{1}{2\ell}} \leq \limsup_{\ell \to \infty} \bigg[\frac{\varepsilon_1(\ell)}{4c^2}\bigg]^{\frac{1}{2\ell}} = \limsup_{\ell \to \infty} \bigg[\frac{\frac{1}{2}\sqrt{\frac{\ell}{\pi}}(\ell-1)4^\ell}{4c^2}\bigg]^{\frac{1}{2\ell}} = 2.
\]
At the same time,
\[
  \frac{1}{2^\ell c^2}\sum_{\substack{\pi \in \mcal{P}_2(2\ell):\\ g(\pi) = 1}} I_\pi^\ell \geq \frac{1}{2^\ell c^2}I_{\pi_{2\ell}},
\]
where $\pi_{2\ell} = (1, 3)(2, 4)(5, 6)(7,8)\cdots(2\ell - 1, 2\ell)$. Note that $\pi_{2\ell}$ is ``one-crossing''. In particular, $g(\pi) = 1$ since $\gamma \circ \pi = (1, 4, 3, 2, 5, 7, 9, \ldots, 2\ell - 1)(6)(8)\cdots(2\ell)$. Moreover, the corresponding graph $\underline{C}_{2\ell}^{\pi_{2\ell}}$ is the star graph with $\ell - 1$ vertices. Since $\underline{C}_{2\ell}^{\pi_{2\ell}}$ is a tree, we know that $I_{\pi_{2\ell}} = 2^{\ell-2}$, whence
\[
  \liminf_{\ell \to \infty} [m_{2\ell}(1,c)]^{\frac{1}{2\ell}} \geq \liminf_{\ell \to \infty} \bigg[\frac{1}{4c^2}\bigg]^{\frac{1}{2\ell}} = 1,
\]
which proves \eqref{eq:bounded_support}. The hypothetical signed measure $\nu_c$ associated to the infinitesimal distribution in the intermediate regime $\lim_{N \to \infty} \frac{b_N}{\sqrt{N}} = c \in (0, \infty)$ then satisfies
\[
  [-1+\varepsilon, 1 - \varepsilon] \not\subset \op{supp}(\nu_c) \subset [-2, 2].
\]
In fact, we expect that $\lim_{\ell \to \infty} [m_{2\ell}(1,c)]^{\frac{1}{2\ell}} = 2$, but we do not prove this here.

\begin{rem}\label{rem:joint_infinitesimal_distribution}
Naturally, one can ask about the joint infinitesimal distribution of independent periodically banded GUE matrices $(\mbf{\Xi}_N^{(i)})_{i \in I}$. To answer this question, we partition the index set $I = I_1 \sqcup I_2$ according to the rates $b_N^{(i)} \to \infty$, where
\[
  I_1 = \{i \in I \mid b_N^{(i)} \gg \sqrt{N}\} \quad \text{and} \quad I_2 = \Big\{i \in I \mathrel{\Big|} \lim_{N \to \infty} \frac{b_N^{(i)}}{\sqrt{N}} \in (0, \infty)\Big\}.
\]
Repeating our banded genus expansion for a mixed trace in $(\mbf{\Xi}_N^{(i)})_{i \in I_1}$ shows that the joint infinitesimal distribution of $(\mbf{\Xi}_N^{(i)})_{i \in I_1}$ is null, which implies that the $(\mbf{\Xi}_N^{(i)})_{i \in I_1}$ are asymptotically infinitesimally free. Similarly, any mixed trace in the two families $(\mbf{\Xi}_N^{(i)})_{i \in I_1}$ and $(\mbf{\Xi}_N^{(i)})_{i \in I_2}$ vanishes in the limit, which implies that $(\mbf{\Xi}_N^{(i)})_{i \in I_1}$ and $(\mbf{\Xi}_N^{(i)})_{i \in I_2}$ are asymptotically infinitesimally free as well.

On the other hand, the same calculation shows that the $(\mbf{\Xi}_N^{(i)})_{i \in I_2}$ are \emph{not} asymptotically infinitesimally free. For example, if $b_N^{(1)}, b_N^{(2)} = \sqrt{N}$, then
\[
  \lim_{N \to \infty} \E[{\Tr}(\mbf{\Xi}_N^{(1)}\mbf{\Xi}_N^{(2)}\mbf{\Xi}_N^{(1)}\mbf{\Xi}_N^{(2)})] = \frac{1}{4},
\]
whereas asymptotic infinitesimal freeness would insist that this limit be zero. Note that our integral interpretation still holds and prescribes a rule for computing the infinitesimal distribution in this case. To account for the possibly different ratios $\lim_{N \to \infty} \frac{b_N^{(i)}}{\sqrt{N}} = c_i$, we must adjust both the integrand and the region of integration via a straightforward combination of the ideas above and \cite[$\S$4.3]{Au18}. We leave the details to the interested reader.

Note that the infinitesimal calculation is very specific (GUE and periodically banded) and cannot be extended to regular band matrices
\[
\mbf{B}_N(i, j) = \indc{|i-j| \leq b_N}.
\]
In particular, let $\mbf{\Xi}_N$ now denote the banded GUE matrix constructed with $\mbf{B}_N$ as above. For $1 \ll b_N \ll N$, we know that $\mu(\mbf{\Xi}_N)$ still converges weakly almost surely to the semicircle distribution \cite{BMP91}. A simple calculation shows that
\[
\E[{\Tr}(\mbf{\Xi}_N^2)] = N - \frac{b_N(b_N + 1)}{2b_N + 1}.
\]
In this case,
\[
\lim_{N \to \infty} \E[{\Tr}(\mbf{\Xi}_N^2)] - N\cdot{\Cat}(1) = \lim_{N \to \infty} - \frac{b_N(b_N + 1)}{2b_N + 1} = -\infty,
\]
and so an infinitesimal distribution does not exist for any such band width.
\end{rem}

\subsection{Finite-rank perturbations}\label{sec:finite_rank_perturbations}
We now consider the multi-matrix model
\[
  \mcal{Z}_N = (\mbf{\Xi}_N^{(i)})_{i \in I}; \qquad
  \mcal{E}_N = (\mbf{E}_N^{(j, k)})_{1 \leq j, k \leq N_0}; \qquad
  \mbf{K}_N = \frac{1}{N}\mbf{J}_N. 
\]
Most of our calculations in this section remain valid in a more general setting. In particular, we extend our definition of
\[
  \mbf{\Xi}_N^{(i)} = \frac{1}{\sqrt{\xi_N^{(i)}}}\mbf{B}_N^{(i)} \circ \mbf{X}_N^{(i)}
\]
to independent unnormalized Wigner matrices $(\mbf{X}_N^{(i)})_{i \in I}$ of the form
\begin{equation}\label{eq:generalized_wigner}
\begin{aligned}
  \E[\mbf{X}_N^{(i)}(j, k)] = 0 \quad \text{and} \quad \E\mathopen{}\big[|\mbf{X}_N^{(i)}(j, k)|^2\big]\mathclose{} = 1, \qquad \forall j < k \in [N];\\
  \sup_{N \in \N} \sup_{i \in I_0} \sup_{j \leq k \in [N]} \E\mathopen{}\big[|\mbf{X}_N^{(i)}(j, k)|^\ell\big]\mathclose{} < \infty, \qquad \forall I_0 \subset I: \#(I_0) < \infty,
\end{aligned}
\end{equation}
where the band widths $(b_N^{(i)})_{i \in I}$ satisfy
\begin{equation}\label{eq:band_width_divergence}
  \lim_{N \to \infty} b_N^{(i)} = \infty, \qquad \forall i \in I.
\end{equation}
Dykema proved that the family $\mcal{W}_N = (\frac{1}{\sqrt{N}}\mbf{X}_N^{(i)})_{i \in I}$ converges in distribution to a semicircular system \cite{Dyk93}. We generalized this result to the family $\mcal{Z}_N$ in \cite{Au18} (recall that if $b_N^{(i)} \geq \lfloor N/2 \rfloor$, then $\mbf{\Xi}_N^{(i)} = \frac{1}{\sqrt{N}}\mbf{X}_N^{(i)}$). For concreteness, we write $(\C\langle \mbf{x} \rangle, \tau_{\mcal{Z}})$ for this limiting distribution, where
\[
  \tau_{\mcal{Z}}(p) = \lim_{N \to \infty} \frac{1}{N}\E[{\Tr}(p(\mcal{Z}_N))], \qquad \forall p \in \C\langle \mbf{x} \rangle.
\]

The results in \cite{Au18} as well as the remainder of this section make use of the traffic probability framework \cite{Mal11}, which we briefly review.

\begin{defn}[Traffic probability]\label{defn:traffic_probability}
By a \emph{multidigraph} $G = (V, E, \source, \target)$, we mean a non-empty set of vertices $V$, a set of edges $E$, and a pair of functions $\source, \target: E \to V$ specifying the source and target of each edge. A \emph{test graph} $T = (G, \gamma)$ is a finite multidigraph $G$ with edge labels $\gamma: E \to I$. For a partition $\pi \in \mcal{P}(V)$, we define $T^\pi = (G^\pi, \gamma^\pi)$ as the test graph obtained from $T$ by identifying the vertices of $T$ according to blocks of $\pi$. Formally, we construct $G^\pi = (V^\pi, E^\pi, \source^\pi, \target^\pi)$ as
\begin{enumerate}[label=(\roman*)]
\item $V^\pi = V/\mathord\sim_\pi$ and $E^\pi = E$;
\item $\source^\pi(e) = [\source(e)]_{\sim_\pi}$ and $\target^\pi(e) = [\target(e)]_{\sim_\pi}$;
\item $\gamma^\pi = \gamma$.
\end{enumerate}
Since $E^\pi = E$, we often omit the superscript and use the same notation for the edge set of the quotient $T^\pi$. We write $\mcal{T}\langle I \rangle$ for the set of all test graphs in $I$ and $\C\mcal{T}\langle I \rangle$ for the complex vector space spanned by $\mcal{T}\langle I \rangle$.

We define the \emph{traffic state} $\tau_N: \C\mcal{T}\langle I \rangle \to \C$ as the unique linear functional
\[
  \tau_N[T] = \frac{1}{N^{c(T)}}\sum_{\phi: V \to [N]} \E\Big[\prod_{e \in E} \mbf{\Xi}_N^{(\gamma(e))}(\phi(\target(e)), \phi(\source(e)))\Big], \qquad \forall T \in \mcal{T}\langle I \rangle,
\]
where $c(T)$ denotes the number of connected components of $T$. For convenience, we abbreviate $(\phi(\target(e)), \phi(\source(e)))$ as $(\phi(e))$. Similarly, we define the \emph{injective traffic state} $\tau_N^0: \C\mcal{T}\langle I \rangle \to \C$ as the unique linear functional
\[
  \tau_N^0[T] = \frac{1}{N^{c(T)}}\sum_{\substack{\phi: V \to [N] \\ \text{s.t. $\phi$ is injective}}} \E\Big[\prod_{e \in E} \mbf{\Xi}_N^{(\gamma(e))}(\phi(e))\Big], \qquad \forall T \in \mcal{T}\langle I \rangle.
\]
Henceforth, we use the notation $\phi: V \hookrightarrow [N]$ to indicate an injective map. The functionals $\tau_N$ and $\tau_N^0$ satisfy the relations
\begin{align*}
  \tau_N[T] &= \sum_{\pi \in \mcal{P}(V)} N^{c(T^\pi)-c(T)}\tau_N^0[T^\pi];\\
  \tau_N^0[T] &= \sum_{\pi \in \mcal{P}(V)} \text{M\"{o}b}(0_V, \pi)N^{c(T^\pi) - c(T)}\tau_N[T^\pi],
\end{align*}
where $0_V$ denotes the singleton partition and $\text{M\"{o}b}$ is the usual M\"{o}bius function on the poset of partitions. 
\end{defn}

\begin{eg}\label{eg:traffic_probability}
Let $p \in \C\langle \mbf{x} \rangle$ be a monomial $p = x_{i(1)} \cdots x_{i(d)}$. Then
\[
  \frac{1}{N}\E[{\Tr}(p(\mcal{Z}_N))] = \tau_N[T_p] = \sum_{\pi \in \mcal{P}(V_p)} \tau_N^0[T_p^\pi],   
\]
where
\[
  \begin{tikzpicture}[shorten > = 2.5pt]
    \node at (-2, 0) {$T_p = $};
    \draw[fill=black] (1,0) circle (1.75pt);
    \draw[fill=black] (-1,0) circle (1.75pt);
    \draw[fill=black] (.5,.866) circle (1.75pt);
    \draw[fill=black] (.5,-.866) circle (1.75pt);
    \draw[fill=black] (-.5,.866) circle (1.75pt);
    \draw[fill=black] (-.5,-.866) circle (1.75pt);
    \draw[semithick, ->] (1,0) to node[pos=.625, right] {\footnotesize$i(d-1)$\normalsize} (.5,-.866);
    \draw[semithick, ->] (.5,-.866) to node[midway,below] {\footnotesize$\cdots$\normalsize} (-.5,-.866);
    \draw[semithick, ->] (-.5,-.866) to node[pos=.375, left] {\footnotesize$i(3)$\normalsize} (-1, 0);
    \draw[semithick, ->] (-1, 0) to node[pos=.625, left] {\footnotesize$i(2)$\normalsize} (-.5, .866);
    \draw[semithick, ->] (-.5, .866) to node[midway, above] {\footnotesize$i(1)$\normalsize} (.5, .866);
    \draw[semithick, ->] (.5, .866) to node[pos=.375, right] {\footnotesize$i(d)$\normalsize} (1, 0);
  \end{tikzpicture}
\]
\end{eg}

We can now prove the following generalization of Lemma 3.2 in \cite{Shl18}.

\begin{lemma}\label{lem:mixed_trace_matrix_units}
For any NC polynomials $p_1, \ldots, p_r \in \C\langle \mbf{x}\rangle$,
\[
  \lim_{N \to \infty} \E\mathopen{}\Big[{\Tr}\mathopen{}\Big(\prod_{s = 1}^r \mbf{E}_N^{(j_{s-1}, k_s)}p_s(\mcal{Z}_N)\Big)\mathclose{}\Big]\mathclose{} = \prod_{s = 1}^r [\indc{k_s = j_s}\tau_{\mcal{Z}}(p_s)],
\]
where $j_0 = j_r$.
\end{lemma}
\begin{proof}
Note that we can rewrite the desired trace as
\begin{align*}
  &{\Tr}\mathopen{}\Big(\mleft[\mbf{E}_N^{(j_r, k_1)}p_1(\mcal{Z}_N)\mbf{E}_N^{(j_1, k_2)}\mright]\mleft[\mbf{E}_N^{(k_2, k_2)}p_2(\mcal{Z}_N)\mbf{E}_N^{(j_2, k_3)}\mright] \cdots \\
  &\mleft[\mbf{E}_N^{(k_{r-1}, k_{r-1})}p_{r-1}(\mcal{Z}_N)\mbf{E}_N^{(j_{r-1}, k_r)}\mright]\mleft[\mbf{E}_N^{(k_r, k_r)}p_r(\mcal{Z}_N)\mbf{E}_N^{(j_r, j_r)}\mright]\Big)\mathclose{}\\
  &= \mathopen{}\Big[\prod_{s = 1}^r [p_s(\mcal{Z}_N)]_{k_s, j_s}\Big]\mathclose{}{\Tr}\mleft(\mbf{E}_N^{(j_r, k_2)}\mbf{E}_N^{(k_2, k_3)} \cdots \mbf{E}_N^{(k_{r-1}, k_r)} \mbf{E}_N^{(k_r, j_r)}\mright) \\
  &= \prod_{s = 1}^r [p_s(\mcal{Z}_N)]_{k_s, j_s},
\end{align*}
which reduces the problem to computing
\[
  \lim_{N \to \infty} \E\mathopen{}\Big[\prod_{s = 1}^r [p_s(\mcal{Z}_N)]_{k_s, j_s}\Big]\mathclose{}.
\]
Furthermore, by linearity, it suffices to prove the result for monomials $p_s \in \C\langle\mbf{x}\rangle$. For concreteness, we write
\[
  p_s = x_{i_s(1)}\cdots x_{i_s(d_s)},
\]
where $i_s: [d_s] \to I$. To convert this to the traffic notation, let $T_s = (G_s, \gamma_s)$ be the test graph
\begin{equation}\label{eq:path_test_graph}
  \begin{tikzpicture}[shorten > = 2.5pt, baseline=(current  bounding  box.center)]
    \node at (-2, 0) {$T_s = $};
    \draw[fill=black] (-1.25,0) circle (1.75pt);
    \draw[fill=black] (-.25,0) circle (1.75pt);
    \draw[fill=black] (.75,0) circle (1.75pt);
    \draw[fill=black] (1.75,0) circle (1.75pt);
    \node at (-1.25, -.375) {\footnotesize$v_{s,0}$\normalsize};
    \node at (-.25, -.375) {\footnotesize$v_{s,1}$\normalsize};
    \node at (.75, -.375) {\footnotesize$\cdots$\normalsize};
    \node at (1.75, -.375) {\footnotesize$v_{s,d_s}$\normalsize};
    \draw[semithick, ->] (-.25,0) to node[midway, above] {\footnotesize$i(1)$\normalsize} (-1.25,0);
    \draw[semithick, ->] (.75,0) to node[midway, above] {\footnotesize$\cdots$\normalsize} (-.25,0);
    \draw[semithick, ->] (1.75,0) to node[midway, above] {\footnotesize$i(d_s)$\normalsize} (.75,0);
  \end{tikzpicture}
\end{equation}
where $V_s = (v_{s, t-1})_{t \in [d_s + 1]}$, $E_s = (e_{s, t})_{t \in [d_s]}$, $v_{s,t-1} \sim_{e_{s, t}} v_{s, t}$, and $\gamma_s(e_{s, t}) = i_s(t)$. We define $T = (G, \gamma) = \sqcup_{s = 1}^r T_s$ as the disjoint union of the $T_s$, in which case
\begin{align*}
  \E\mathopen{}\Big[\prod_{s = 1}^r [p_s(\mcal{Z}_N)]_{k_s, j_s}\Big]\mathclose{} &= \sum_{\substack{\phi: V \to [N] \text{ s.t.}\\ \phi(v_{s, 0}) = k_s \text{ and}\\ \phi(v_{s, d_s}) = j_s}} \E\mathopen{}\Big[\prod_{e \in E} \mbf{\Xi}_N^{(\gamma(e))}(\phi(e))\Big]\mathclose{} \\
  &= \sum_{\pi \in \mcal{P}(V)} \sum_{\substack{\phi: V^\pi \hookrightarrow [N] \text{ s.t.}\\ \phi([v_{s, 0}]_{\sim_\pi}) = k_s \text{ and}\\ \phi([v_{s, d_s}]_{\sim_\pi}) = j_s}} \E\mathopen{}\Big[\prod_{e \in E} \mbf{\Xi}_N^{(\gamma(e))}(\phi(e))\Big]\mathclose{},
\end{align*}
where we recall that $E^\pi = E$. Note that the inner sum on the previous line might be empty: for example, if $v_{s, 0} \sim_{\pi} v_{s', d_s'}$ for some $k_s \neq j_{s'}$. Conversely, if $k_s = j_s'$ for some $s, s' \in [r]$, then we must have $v_{s, 0} \sim_\pi v_{s', d_{s'}}$. Thus, taking into account the various indices, we can restrict the outer summation over $\mcal{P}(V)$ to
\[
  \mcal{P}_+(V) = \mleft\{\pi \in \mcal{P}(V) \mathrel{}\mathclose{}\middle|\mathopen{}\mathrel{}  \begin{aligned}v_{s, 0} &\sim_\pi v_{s', d_{s'}} &\text{iff } k_s &= j_{s'}\\v_{s, 0} &\sim_\pi v_{s', 0} &\text{iff } k_s &= k_{s'}\\v_{s, d_s} &\sim_\pi v_{s', d_{s'}} &\text{iff \hspace{1pt}} j_s &= j_{s'}\end{aligned}\mright\}.
\]

We now analyze the contribution from a quotient $T^\pi$. First, we decompose $T^\pi$ into its connected components $T^\pi = T_{(1)}^\pi \sqcup \cdots \sqcup T_{(u)}^\pi$, where the notation $T_{(n)}^\pi$ is meant to distinguish between the test graphs $T_{(n)}^\pi$ and $T_s$. For an injective map $\phi: V^\pi \hookrightarrow [N]$, the independence of our matrix entries implies that
\[
  \E\mathopen{}\Big[\prod_{e \in E} \mbf{\Xi}_N^{(\gamma(e))}(\phi(e))\Big]\mathclose{} = \prod_{n = 1}^u \E\mathopen{}\Big[\prod_{e \in E_{(n)}^\pi} \mbf{\Xi}_N^{(\gamma(e))}(\phi(e))\Big]\mathclose{}.
\]
Let us then focus on a connected component $T_{(n)}^\pi = (G_{(n)}^\pi, \gamma_{(n)}^\pi)$. We define $\mcal{L}_{(n)}^\pi$ as the set of loop edges of $T_{(n)}^\pi$, which divides $E_{(n)}^\pi = \mcal{L}_{(n)}^\pi \sqcup \mcal{N}_{(n)}^\pi$. As before, we write $\underline{G}_{(n)}^\pi = (\underline{V}_{(n)}^\pi, \underline{E}_{(n)}^\pi)$ for the underlying simple graph. For an edge $e \in E_{(n)}^\pi$, we define
\begin{align*}
  [e] &= \{e' \in E_{(n)}^\pi \mid \{\source(e), \target(e)\} = \{\source(e'), \target(e')\}\};\\
  [e]_i &= \{e' \in [e] \mid \gamma(e') = i\}.
\end{align*}
Naturally, we can think of $\underline{E}_{(n)}^\pi = \{[e] : e \in \mcal{N}_{(n)}^\pi\}$. By a slight abuse of notation, we also write $\underline{\mcal{L}}_{(n)}^\pi = \{[l] : l \in \mcal{L}_{(n)}^\pi\}$. Separating the normalization
\begin{equation}\label{eq:separated}
  \E\mathopen{}\Big[\prod_{e \in E_{(n)}^\pi} \mbf{\Xi}_N^{(\gamma(e))}(\phi(e))\Big]\mathclose{} = \mathopen{}\Bigg(\prod_{e \in E_{(n)}^\pi} \frac{1}{\sqrt{\xi_N^{(\gamma(e))}}}\Bigg)\mathclose{}\E\mathopen{}\Big[\prod_{e \in E_{(n)}^\pi} \mbf{X}_N^{(\gamma(e))}(\phi(e))\Big]\mathclose{},
\end{equation}
we can again use the injectivity of $\phi$ to decompose the remaining expectation as
\begin{equation}\label{eq:uniform_bound}
  \prod_{[l] \in \underline{\mcal{L}}_{(n)}^\pi} \E\mathopen{}\Big[\prod_{l' \in [l]} \mbf{X}_N^{(\gamma(l'))}(\phi(l'))\Big]\mathclose{} \prod_{[e] \in \underline{E}_{(n)}^\pi} \E\mathopen{}\Big[\prod_{e' \in [e]} \mbf{X}_N^{(\gamma(e'))}(\phi(e'))\Big]\mathclose{} = O_d(1).
\end{equation}
The asymptotic follows from our strong moment assumption \eqref{eq:generalized_wigner}, which bounds the contribution from such a term uniformly in $\pi$ and $\phi$, where $d = \sum_{s=1}^r d_s$ is the total degree of our monomials $p_s$. Strictly speaking, the asymptotic depends on both $d$ and the finite set $I_0 = \gamma(E)$, but both are fixed independent of $N$ by our monomials $p_s$. For convenience, we omit this last detail from the notation.

We would then like to bound the number of injective maps $\phi$ that actually contribute (i.e., the number of $\phi$ such that the term in \eqref{eq:uniform_bound} is non-zero). Note that since the off-diagonal entries of our matrices are centered, we can assume that
\begin{equation}\label{eq:edges_multiplicity}
  \#([e]_{\gamma(e)}) \geq 2, \qquad \forall e \in \underline{E}_{(n)}^\pi; 
\end{equation}
otherwise, one of the factors in the product above vanishes. Of course, the graph $\underline{G}_{(n)}^\pi$ is still connected with the same vertex set as $G_{(n)}^\pi$, whence
\begin{equation}\label{eq:tree_formula}
  \#(V_{(n)}^\pi) \leq \#(\underline{E}_{(n)}^\pi) + 1.
\end{equation}

We must also remember to include the band width constraint in our bound. In particular, a contributing map $\phi$ satisfies
\[
  |\phi(\source(e)) - \phi(\target(e))|_N \leq \min_{e' \in [e]} b_N^{(\gamma(e'))}, \qquad \forall e \in \mcal{N}_{(n)}^\pi.
\]
We introduce some notation for the set of admissible maps
\[
  A_{N, \pi} = \mleft\{\phi: V^\pi \hookrightarrow [N] \mathrel{}\mathclose{}\middle|\mathopen{}\mathrel{} \begin{aligned} \phi([v_{s, 0}]_{\sim_\pi}) = k_s \quad \text{and} \quad \phi([v_{s, d_s}]_{\sim_\pi}) = j_s\\|\phi(\source(e)) - \phi(\target(e))|_N \leq b_N^{(\gamma(e))}, \quad \forall e \in E\end{aligned}\mright\}.
\]
Similarly, we define
\[
  A_{N, \pi}^{(n)} = \mleft\{\phi_n: V_{(n)}^\pi \hookrightarrow [N] \mathrel{}\mathclose{}\middle|\mathopen{}\mathrel{} \begin{aligned}\begin{aligned} \phi_n([v_{s, 0}]_{\sim_\pi}) &= k_s & &\text{if } [v_{s, 0}]_{\sim_\pi} \in V_{(n)}^\pi \\ \phi_n([v_{s, d_s}]_{\sim_\pi}) &= j_s & &\text{if } [v_{s, d_s}]_{\sim_\pi} \in V_{(n)}^\pi\end{aligned}\\|\phi_n(\source(e)) - \phi_n(\target(e))|_N \leq b_N^{(\gamma(e))},\quad \forall e \in E_{(n)}^\pi\end{aligned}\mright\}.
\]
Note that
\begin{equation}\label{eq:submultiplicative_maps}
  \#(A_{N, \pi}) \leq \prod_{n = 1}^u \#(A_{N, \pi}^{(n)}).
\end{equation}

Consider a spanning tree $\underline{H}_{(n)}^\pi = (\underline{V}_{(n)}^\pi, \underline{F}_{(n)}^\pi)$ of $\underline{G}_{(n)}^\pi$. We think of a spanning tree as recording a minimal working subset of the band width constraints. In particular, $\underline{H}_{(n)}^\pi$ bounds the number of contributing maps $\phi|_{V_{(n)}^\pi} \in A_{N, \pi}^{(n)}$ by
\begin{equation}\label{eq:weak_bound_components}
  \#(A_{N, \pi}^{(n)}) \leq N\prod_{[e] \in \underline{F}_{(n)}^\pi} \min_{e' \in [e]} \xi_N^{(\gamma(e'))}.
\end{equation}
To see this, pick an arbitrary initial vertex $[v_0]_{\sim \pi}$ of $\underline{H}_{(n)}^\pi$. Clearly, we have $N$ options for $\phi([v_0]_{\sim_\pi}) \in [N]$ at this stage. The bound then follow from walking through the rest of our graph while satisfying the band width constraints imposed by the edges $[e] \in \underline{F}_{(n)}^\pi$. Note that this fails to account for the special vertices $[v_{s, 0}]_{\sim \pi}$ and $[v_{s, d_s}]_{\sim_\pi}$, which have fixed labels $\phi([v_{s, 0}]_{\sim \pi}) = k_s$ and $\phi([v_{s, d_s}]_{\sim \pi}) = j_s$ respectively. In particular, each connected component has at least one such special vertex. Choosing this special vertex to be the initial vertex $[v_0]_{\sim \pi}$ removes the factor of $N$ in our earlier bound, and so
\begin{equation}\label{eq:bound_components}
  \#(A_{N, \pi}^{(n)}) \leq \prod_{[e] \in \underline{F}_{(n)}^\pi} \min_{e' \in [e]} \xi_N^{(\gamma(e'))} = O\mathopen{}\Big(\prod_{e \in E_{(n)}^\pi} \sqrt{\xi_N^{(\gamma(e))}}\Big)\mathclose{},
\end{equation}
where the asymptotic follows from \eqref{eq:edges_multiplicity}. In view of \eqref{eq:separated}-\eqref{eq:bound_components}, we conclude that
\begin{align}\label{eq:expectation_asymptotic}
  \E\mathopen{}\Big[\prod_{s = 1}^r [p_s(\mcal{Z}_N)]_{k_s, j_s}\Big]\mathclose{} &= \sum_{\pi \in \mcal{P}_+(V)} \sum_{\phi \in A_{N, \pi}} \prod_{n = 1}^u \mathopen{}\Bigg[\frac{\E\mathopen{}\big[\prod_{e \in E_{(n)}^\pi} \mbf{X}_N^{(\gamma(e))}(\phi(e))\big]\mathclose{}}{\prod_{e \in E_{(n)}^\pi} \sqrt{\xi_N^{(\gamma(e))}}}\Bigg]\mathclose{}\\
                                                      &= \sum_{\pi \in \mcal{P}_+(V)} \prod_{n = 1}^u O_d\mathopen{}\Bigg(\frac{\#(A_{N, \pi}^{(n)})}{\prod_{e \in E_{(n)}^\pi} \sqrt{\xi_N^{(\gamma(e))}}}\Bigg)\mathclose{} = O_d(1).\notag
\end{align}

Altogether, our analysis implies that the expectation survives the normalization, but only just barely. Indeed, in formulating the bound \eqref{eq:bound_components}, we only considered one of the special vertices $[v_{s, 0}]_{\sim_\pi}, [v_{s, d_s}]_{\sim_\pi}$ despite the fact that each test graph $T_s$ has two such vertices $v_{s, 0}, v_{s, d_s}$ before the identifications by $\pi \in \mcal{P}_+(V)$. Assume then that $k_s \neq j_s$ for some $s \in [r]$. In this case, $[v_{s, 0}]_{\sim_\pi} \neq [v_{s, d_s}]_{\sim_\pi}$ are distinct vertices in some connected component $T_{(n)}^\pi$. As a result, we lose an additional degree of freedom when choosing a contributing map $\phi|_{V_{(n)}^\pi} \in A_{N, \pi}^{(n)}$. To see this, we return to our spanning tree $\underline{H}_{(n)}^\pi$. We denote the last edge on the unique path from $[v_{s, 0}]_{\sim_\pi}$ to $[v_{s, d_s}]_{\sim_\pi}$ in $\underline{H}_{(n)}^\pi$ by $[e_*]$. Running through the same argument as before with $[v_{s, 0}]_{\sim_\pi}$ as the initial vertex now gives the improved bound
\begin{equation}\label{eq:bound_components_removed}
  \#(A_{N, \pi}^{(n)}) \leq \prod_{[e] \in \underline{F}_{(n)}^\pi \setminus \{[e_*]\}} \min_{e' \in [e]} \xi_N^{(\gamma(e'))} = o\mathopen{}\Big(\prod_{[e] \in \underline{F}_{(n)}^\pi} \min_{e' \in [e]} \xi_N^{(\gamma(e'))}\Big)\mathclose{},
\end{equation}
where the asymptotic follows from \eqref{eq:band_width_divergence}. Putting this back in to \eqref{eq:expectation_asymptotic} proves that
\[
  \lim_{N \to \infty} \E\mathopen{}\Big[\prod_{s = 1}^r [p_s(\mcal{Z}_N)]_{k_s, j_s}\Big]\mathclose{} = 0 \text{ if } k_s \neq j_s \text{ for some } s \in [r]. 
\]

Let us now assume that $k_s = j_s$ for every $s \in [r]$. A partition $\pi \in \mcal{P}_+(V)$ then necessarily identifies $v_{s, 0} \sim_\pi v_{s, d_s}$ for every $s \in [r]$. Furthermore, if $k_s = k_{s'}$ for some $s, s' \in [r]$, then $\pi$ must also identify $v_{s, 0} \sim_\pi v_{s', 0}$. We imagine making these identifications first before carrying out the rest of the identifications prescribed by $\pi$. At the first step, this corresponds to identifying the ends of the test graph \eqref{eq:path_test_graph}, creating a directed cycle $C_s$ with a special vertex $[v_{s, 0}]_{\sim_\pi} = [v_{s, d_s}]_{\sim_\pi}$ that we can think of as a root. We then identify the roots of different cycles $C_s, C_{s'}$ if $k_s = k_{s'}$. It will be convenient to redefine \eqref{eq:path_test_graph} to account for this first step beforehand, namely
\begin{equation}\label{eq:cycle_test_graph}
  \begin{tikzpicture}[shorten > = 2.5pt, baseline=(current  bounding  box.center)]
    \node at (-2.5, 0) {$T_s = $};
    \draw[fill=black] (1,0) circle (1.75pt);
    \node at (1.625, 0) {\footnotesize$v_{s, d_s-1}$\normalsize};
    \draw[fill=black] (-1,0) circle (1.75pt);
    \node at (-1.375, 0) {\footnotesize$v_{s, 2}$\normalsize};
    \draw[fill=black] (.5,.866) circle (1.75pt);
    \node at (1, 1.091) {\footnotesize$v_{s, d_s\phantom{-2}}$\normalsize}; 
    \draw[fill=black] (.5,-.866) circle (1.75pt);
    \node at (1, -1.166) {\footnotesize$v_{s, d_s-2}$\normalsize}; 
    \draw[fill=black] (-.5, .866) circle (1.75pt);
    \node at (-.725, -1.166) {\footnotesize$v_{s, 3}$\normalsize};     
    \draw[fill=black] (-.5,-.866) circle (1.75pt);
    \node at (-.725,1.091) {\footnotesize$v_{s, 1}$\normalsize}; 
    \draw[semithick, ->] (1,0) to node[pos=.625, right] {\footnotesize$i(d_s-1)$\normalsize} (.5,-.866);
    \draw[semithick, ->] (.5,-.866) to node[midway,below] {\footnotesize$\cdots$\normalsize} (-.5,-.866);
    \draw[semithick, ->] (-.5,-.866) to node[pos=.375, left] {\footnotesize$i(3)$\normalsize} (-1, 0);
    \draw[semithick, ->] (-1, 0) to node[pos=.625, left] {\footnotesize$i(2)$\normalsize} (-.5, .866);
    \draw[semithick, ->] (-.5, .866) to node[midway, above] {\footnotesize$i(1)$\normalsize} (.5, .866);
    \draw[semithick, ->] (.5, .866) to node[pos=.375, right] {\footnotesize$i(d_s)$\normalsize} (1, 0);
  \end{tikzpicture}
\end{equation}

Now, suppose that $\pi \in \mcal{P}_+(V)$ identifies vertices across different cycles:
\[
  v_{s, t} \sim_\pi v_{s', t'} \text{ for some } s \neq s'.
\]
If $k_s \neq k_{s'}$, then we claim that
\begin{equation}\label{eq:null_sum}
  \lim_{N \to \infty} \sum_{\substack{\phi: V^\pi \hookrightarrow [N] \text{ s.t.}\\ \phi([v_{s, d_s}]_{\sim_\pi}) = k_s}} \E\mathopen{}\Big[\prod_{e \in E} \mbf{\Xi}_N^{(\gamma(e))}(\phi(e))\Big]\mathclose{} = 0.
\end{equation}
To see this, let $T_{(n_*)}^\pi$ denote the connected component of $T^\pi$ that contains the vertex $[v_{s, d_s}]_{\sim \pi}$. By assumption, $T_{(n_*)}^\pi$ also contains the vertex $[v_{s', d_{s'}}]_{\sim \pi} \neq [v_{s, d_s}]_{\sim \pi}$. Our earlier work shows that $\#(A_{N, \pi}^{(n_*)})$ satisfies the asymptotic \eqref{eq:bound_components_removed} since the component $T_{(n_*)}^\pi$ has two special vertices to account for, which proves \eqref{eq:null_sum}.

If $k_s = k_{s'}$, then we claim that \eqref{eq:null_sum} holds if $v_{s, t} \not\sim_\pi v_{s', d_{s'}}$. To see this, let $T_{(n_*)}^\pi$ be as before. If $v_{s, t} \not\sim_\pi v_{s', d_{s'}}$, then $[v_{s, t}]_{\sim_\pi} \neq [v_{s', d_{s'}}]_{\sim_\pi}$ are distinct vertices in $T_{(n_*)}^\pi$. In particular, there are four edge-disjoint paths from $[v_{s, t}]_{\sim_\pi}$ to $[v_{s', d_{s'}}]_{\sim_\pi}$. Indeed, there are two edge-disjoint paths from $[v_{s, t}]_{\sim_\pi}$ to $[v_{s, d_s}]_{\sim_\pi}$ using only the edges of $T_s$, and there are two edge-disjoint paths from $[v_{s', t'}]_{\sim_\pi}$ to $[v_{s', d_{s'}}]_{\sim_\pi}$ using only the edges of $T_{s'}$. Thus, any spanning tree $\underline{H}_{(n_*)}^\pi$ of $T_{(n_*)}^\pi$ will necessarily omit (at least) one of the total edges from these paths. In view of \eqref{eq:edges_multiplicity} and \eqref{eq:band_width_divergence}, we conclude that
\[
  \#(A_{N, \pi}^{(n_*)}) \leq \prod_{[e] \in \underline{F}_{(n_*)}^\pi} \min_{e' \in [e]} \xi_N^{(\gamma(e'))} = o\mathopen{}\Big(\prod_{e \in E_{(n_*)}^\pi} \sqrt{\xi_N^{(\gamma(e))}}\Big)\mathclose{},
\]
which again proves \eqref{eq:null_sum}.

Thus, we are left to consider partitions
\begin{equation}\label{eq:separated_partitions}
  \mcal{P}_{++}(V) = \mleft\{\pi \in \mcal{P}_+(V) \mathrel{}\mathclose{}\middle|\mathopen{}\mathrel{}  \begin{aligned}v_{s, t} &\sim_\pi v_{s', t'} & &\text{only if } k_s=k_s'\\v_{s, t} &\sim_\pi v_{s', d_{s'}} & &\text{if } v_{s, t} \sim_\pi v_{s', t'} \text{ and } s \neq s'\end{aligned}\mright\},
\end{equation}
where
\[
  \lim_{N \to \infty} \E\mathopen{}\Big[\prod_{s = 1}^r [p_s(\mcal{Z}_N)]_{k_s, k_s}\Big]\mathclose{} = \lim_{N \to \infty} \sum_{\pi \in \mcal{P}_{++}(V)} \sum_{\phi \in A_{N, \pi}} \prod_{n = 1}^u \E\mathopen{}\Big[\prod_{e \in E_{(n)}^\pi} \mbf{\Xi}_N^{(\gamma(e))}(\phi(e))\Big]\mathclose{}.
\]
Note that $\mcal{P}_{++}(V)$ factorizes into partitions of the test graphs \eqref{eq:cycle_test_graph} via the bijection
\[
  \bigtimes_{s = 1}^r \mcal{P}(V_s) \to \mcal{P}_{++}(V), \qquad (\pi_1, \ldots, \pi_r) \mapsto \coprod_{s=1}^r \pi_s,
\]
where $\pi = \coprod_{s = 1}^r \pi_s$ is the partition obtained from $(\pi_1, \ldots, \pi_r)$ by first taking the disjoint union of the blocks of the $\pi_s$ and then identifying the vertices $v_{s, d_s} \sim_\pi v_{s', d_{s'}}$ that satisfy $k_s = k_{s'}$. Of course, the resulting quotient test graph $T^\pi$ might have fewer than $r$ connected components; however, the defining property \eqref{eq:separated_partitions} of $\mcal{P}_{++}(V)$ implies that $T^\pi$ can be obtained as follows: first, let $(\pi_1, \ldots, \pi_r)$ be the factorization of $\pi$ as above. Next, apply the partitions $\pi_s$ to obtain the quotient test graphs $T_s^{\pi_s} = (V_s^{\pi_s}, E_s^{\pi_s})$. Finally, in the disjoint union of the $T_s^{\pi_s}$, identify the vertices $[v_{s, d_s}]_{\sim_{\pi_s}}$ and $[v_{s', d_{s'}}]_{\sim_{\pi_{s'}}}$ if $k_s = k_{s'}$. The injectivity of the maps $\phi \in A_{N, \pi}$ then implies that
\begin{align*}
  \E\mathopen{}\Big[\prod_{s = 1}^r [p_s(\mcal{Z}_N)]_{k_s, k_s}\Big]\mathclose{} &= \sum_{\pi \in \mcal{P}_{++}(V)} \sum_{\phi \in A_{N, \pi}} \prod_{n = 1}^u \E\mathopen{}\Big[\prod_{e \in E_{(n)}^\pi} \mbf{\Xi}_N^{(\gamma(e))}(\phi(e))\Big]\mathclose{} \\
                                                             &= \sum_{(\pi_1, \ldots, \pi_r) \in \times_{r=1}^s \mcal{P}(V_s)} \sum_{\phi \in A_{N, \coprod_{s=1}^r \pi_s}} \prod_{s = 1}^r \E\mathopen{}\Big[\prod_{e \in E_s^{\pi_s}} \mbf{\Xi}_N^{(\gamma(e))}(\phi(e))\Big]\mathclose{}.
\end{align*}

We would also like to factorize the set
\[
  A_{N, \coprod_{s=1}^r \pi_s} \to \bigtimes_{s = 1}^r B_{N, \pi_s}, \qquad \phi \mapsto (\phi|_{V_1^{\pi_1}}, \ldots, \phi|_{V_s^{\pi_s}}),
\]
where
\[
  B_{N, \pi_s} = \mleft\{\phi_s: V_s^{\pi_s} \hookrightarrow [N] \mathrel{\Bigg|} \begin{aligned} \phi_s([v_{s, d_s}]_{\sim_{\pi_s}}) = k_s \\ |\phi_s(\source(e)) - \phi_s(\target(e))|_N \leq b_N^{(\gamma(e))}, \quad \forall e \in E_s\end{aligned}\mright\};
\]
however, in general, this map is not bijective since $\#(A_{N, \coprod_{s=1}^r \pi_s}) < \prod_{s = 1}^r \#(B_{N, \pi_s})$ for $r \geq 2$. Nevertheless, we do have the asymptotic equality
\begin{equation}\label{eq:factorizing_maps_asymptotically}
  \lim_{N \to \infty} \frac{ \#(A_{N, \coprod_{s=1}^r \pi_s})}{\prod_{s = 1}^r \#(B_{N, \pi_s})} = 1.
\end{equation}
The contributions from the additional terms counted by the maps in $\times_{s = 1}^r B_{N, \pi_s}$ can still be bounded uniformly via \eqref{eq:uniform_bound}. In view of \eqref{eq:factorizing_maps_asymptotically}, this implies that such overcounting will not affect our calculations in the limit. In other words,
\begin{gather*}
  \lim_{N \to \infty} \sum_{(\pi_1, \ldots, \pi_r) \in \times_{r=1}^s \mcal{P}(V_s)} \sum_{\phi \in A_{N, \coprod_{s=1}^r \pi_s}} \prod_{s = 1}^r \E\mathopen{}\Big[\prod_{e \in E_s^{\pi_s}} \mbf{\Xi}_N^{(\gamma(e))}(\phi(e))\Big]\mathclose{} \\
  = \lim_{N \to \infty} \prod_{s = 1}^r \sum_{\pi_s \in \mcal{P}(V_s)} \sum_{\phi_s \in B_{N, \pi_s}}  \E\mathopen{}\Big[\prod_{e \in E_s^{\pi_s}} \mbf{\Xi}_N^{(\gamma(e))}(\phi_s(e))\Big]\mathclose{}.
\end{gather*}
So, we will be done if we can prove that
\begin{equation}\label{eq:matrix_unit_normalization}
  \lim_{N \to \infty} \sum_{\pi_s \in \mcal{P}(V_s)} \sum_{\phi_s \in B_{N, \pi_s}}  \E\mathopen{}\Big[\prod_{e \in E_s^{\pi_s}} \mbf{\Xi}_N^{(\gamma(e))}(\phi_s(e))\Big]\mathclose{} = \lim_{N \to \infty} \frac{1}{N}\E[{\Tr}(p_s(\mcal{Z}_N))] .
\end{equation}

The main result in \cite{Au18} implies that
\[
  \lim_{N \to \infty} \frac{1}{N}\E[{\Tr}(p_s(\mcal{Z}_N))] = \lim_{N \to \infty} \sum_{\substack{\pi_s \in \mcal{P}(V_s) \text{ s.t. } \\ T_s^{\pi_s} \text{ is a colored} \\ \text{double tree}}} \frac{1}{N} \sum_{\phi_s \in \mcal{B}_{N, \pi_s}^{(s)}} \E\mathopen{}\Big[\prod_{e \in E_s^{\pi_s}} \mbf{\Xi}_N^{(\gamma(e))}(\phi_s(e))\Big]\mathclose{},
\]
where a colored double tree is a double tree whose twin edges $[e] = \{e, e'\}$ each have the same color $\gamma(e) = \gamma(e')$ and
\[
  \mcal{B}_{N, \pi_s} = \mleft\{\phi_s: V_s^{\pi_s} \hookrightarrow [N] \mathrel{}\mathclose{}\middle|\mathopen{}\mathrel{} |\phi_s(\source(e)) - \phi_s(\target(e))|_N \leq b_N^{(\gamma(e))}, \quad \forall e \in E_s \mright\}.
\]
In short, this follows from \eqref{eq:edges_multiplicity}, \eqref{eq:tree_formula}, and the spanning tree argument. Similarly, we can strict the outer sum in \eqref{eq:matrix_unit_normalization} to the same class of partitions $\pi_s$. Note that for large $N$, the periodicity of the band width condition implies that
\[
  \frac{\#(\mcal{B}_{N, \pi_s})}{N} = \#(B_{N, \pi_s}).
\]
We use the fact that a quotient of a directed cycle is a double tree only if each of its twin edges $[e] = \{e, e'\}$ go in opposite directions $\source(e) = \target(e')$ and $\source(e') = \target(e)$ \cite[Figure 5]{Au18}. In that case,
\[
  \E\mathopen{}\Big[\prod_{e \in E_s^{\pi_s}} \mbf{X}_N^{(\gamma(e))}(\phi_s(e))\Big]\mathclose{} = 1
\]
since the calculations in the expectation only involve variances (as opposed to pseudo-variances). This homogeneity allows us to conclude that averaging over the labels $\phi_s([v_{s, d_s}]_{\sim_{\pi_s}}) \in [N]$ does not affect the calculation. Consequently,
\[
  \lim_{N \to \infty} \frac{1}{N} \sum_{\phi_s \in \mcal{B}_{N, \pi_s}} \E\mathopen{}\Big[\prod_{e \in E_s^{\pi_s}} \mbf{\Xi}_N^{(\gamma(e))}(\phi_s(e))\Big]\mathclose{} = \lim_{N \to \infty} \sum_{\phi_s \in B_{N, \pi_s}} \E\mathopen{}\Big[\prod_{e \in E_s^{\pi_s}} \mbf{\Xi}_N^{(\gamma(e))}(\phi_s(e))\Big]\mathclose{},
\]
as was to be shown.
\end{proof}

Assuming an infinitesimal distribution for the family $\mcal{Z}_N$, Lemma \ref{lem:mixed_trace_matrix_units} proves that $\mcal{Z}_N$ and $\mcal{E}_N$ are asymptotically infinitesimally free. Indeed, this follows from a straightforward application of the following criteria for infinitesimal freeness.

\begin{prop}[\hspace{-1pt}\cite{Shl18}]\label{prop:criteria_infinitesimal_freeness}
Let $(\mcal{A}, \varphi, \varphi')$ be a tracial infinitesimal NC probability space. Suppose that $\mcal{Z}$ and $\mcal{E}$ are subalgebras of $\mcal{A}$ such that $\mcal{E} \subset \ker(\varphi)$ (in particular, $\mcal{E}$ is non-unital). Then $\mcal{Z}$ and $\mcal{E}$ are infinitesimally free iff for any $r$-tuples $(E_s)_{s = 1}^r \subset \mcal{E}$ and $(Z_s)_{s = 1}^r \subset \interior{\mcal{Z}}$, we have the identities
\begin{enumerate}[label=(\roman*)]
\item $\varphi(E_1Z_1E_2Z_2 \cdots E_rZ_r) = 0$;
\item $\varphi'(E_1Z_1E_2Z_2 \cdots E_r Z_r) = 0$.
\end{enumerate}
\end{prop}

For example, this proves a preliminary version of Theorem \ref{thm:infinitesimal_freeness_wigner} (\emph{resp.,} Theorem \ref{thm:infinitesimal_freeness_band}) restricted to the matrices $\mbf{W}_N$ (\emph{resp.,} $\mbf{\Xi}_N$) and $(\mbf{E}_N^{(j, k)})_{1 \leq j, k \leq N_0}$. We now extend the calculation to include the matrix $\mbf{K}_N = \frac{1}{N}\mbf{J}_N$. For this, we will need the following lemma concerning the formation of double trees as quotients of paths.

\begin{lemma}\label{lem:paths_double_trees}
Let $G_n = (V_n, E_n)$ be a path graph of length $n$, where
\[
  \begin{tikzpicture}
    \node at (-2, 0) {$G_n = $};
    \draw[fill=black] (-1.25,0) circle (1.75pt);
    \draw[fill=black] (-.25,0) circle (1.75pt);
    \draw[fill=black] (.75,0) circle (1.75pt);
    \draw[fill=black] (1.75,0) circle (1.75pt);
    \node at (-1.25, -.375) {\footnotesize$v_0$\normalsize};
    \node at (-.25, -.375) {\footnotesize$v_1$\normalsize};
    \node at (.75, -.375) {\footnotesize$\cdots$\normalsize};
    \node at (1.75, -.375) {\footnotesize$v_n$\normalsize};
    \draw[semithick] (-.25,0) to node[midway, above] {\footnotesize$e_1$\normalsize} (-1.25,0);
    \draw[semithick] (.75,0) to node[midway, above] {\footnotesize$\cdots$\normalsize} (-.25,0);
    \draw[semithick] (1.75,0) to node[midway, above] {\footnotesize$e_n$\normalsize} (.75,0);
  \end{tikzpicture}  
\]
If $\pi \in \mcal{P}(V_n)$ is such that $G_n^\pi$ is a double tree, then $v_0 \sim_\pi v_n$.
\end{lemma}
\begin{proof}
Since a double tree has an even number of edges, we only need to prove the result for even values of $n$. We proceed by induction on the length of the path. If $n \in \{0, 2\}$, then the statement follows. So, assume the result is true for paths of length $n \leq 2m$, and consider $G_{2m+2}$. If $G_{2m+2}^\pi$ is a double tree, then it must identify $v_0$ with another vertex $v_i \in V_{2m+2}$ for some $i \in [2m+2]$. Indeed, this follows from the fact that the degree of every vertex in a double tree is even. The edges $e_1, \ldots, e_i$ then form a trail in $G_{2m + 2}^\pi$ starting and ending at the same vertex $v_0 \sim_\pi v_i$. This implies that the subgraph $H$ spanned by these edges is also a double tree. Since the remaining edges $e_{i+1}, \ldots, e_n$ span a connected subgraph $K$ of $G_{2m + 2}^\pi$, the fact that $H$ is a double tree implies that $K$ is a double tree as well. We can then apply the induction hypothesis to conclude that $v_i \sim_\pi v_{2m+2}$.
\end{proof}

We use this to prove the analogue of Lemma \ref{lem:mixed_trace_matrix_units} for $\mbf{K}_N$.

\begin{lemma}\label{lem:mixed_trace_all_ones}
For any NC polynomials $p_1, \ldots, p_r \in \C\langle \mbf{x}\rangle$,
\[
  \lim_{N \to \infty} \E\mathopen{}\Big[{\Tr}\mathopen{}\Big(\prod_{s = 1}^r \mbf{K}_Np_s(\mcal{Z}_N)\Big)\mathclose{}\Big]\mathclose{} = \prod_{s = 1}^r \tau_{\mcal{Z}}(p_s).
\]
\end{lemma}
\begin{proof}
We carry forward the notation from the proof of Lemma \ref{lem:mixed_trace_matrix_units}. In particular, restricting to monomials $p_s$, we write $T_s$ for the test graphs \eqref{eq:path_test_graph}; $T$ for their disjoint union; and $(T_{(n)}^\pi)_{n = 1}^u$ for the connected components of a quotient $T^\pi$. We redefine the set of admissible maps since we no longer have special vertices with fixed labels to account for, namely
\begin{align*}
  A_{N, \pi} &= \mleft\{\phi: V^\pi \hookrightarrow [N] \mathrel{}\mathclose{}\middle|\mathopen{}\mathrel{} |\phi(\source(e)) - \phi(\target(e))|_N \leq b_N^{(\gamma(e))}, \quad \forall e \in E\mright\};\\
  A_{N, \pi}^{(n)} &= \mleft\{\phi_n: V_{(n)}^\pi \hookrightarrow [N] \mathrel{}\mathclose{}\middle|\mathopen{}\mathrel{} |\phi_n(\source(e)) - \phi_n(\target(e))|_N \leq b_N^{(\gamma(e))},\quad \forall e \in E_{(n)}^\pi\mright\}.
\end{align*}
We still have the bounds \eqref{eq:submultiplicative_maps} and \eqref{eq:weak_bound_components}, which imply the following analogue of \eqref{eq:expectation_asymptotic}:
\begin{align}\label{eq:expectation_asymptotic_all_ones}
  \E\mathopen{}\Big[{\Tr}\mathopen{}\Big(\prod_{s = 1}^r \mbf{K}_Np_s(\mcal{Z}_N)\Big)\mathclose{}\Big]\mathclose{} &= \sum_{\pi \in \mcal{P}(V)} \frac{1}{N^{r-u}} \prod_{n = 1}^u O_d\mathopen{}\Bigg(\frac{\#(A_{N, \pi}^{(n)})}{N\prod_{e \in E_{(n)}^\pi} \sqrt{\xi_N^{(\gamma(e))}}}\Bigg)\mathclose{} \\
  &= \sum_{\pi \in \mcal{P}(V)} O_d(N^{u-r}).\notag
\end{align}
Thus, we can restrict to partitions $\pi \in \mcal{P}(V)$ such that $T^\pi$ has exactly $r$ connected components. Of course, since $T$ already has $r$ connected components, this means that we are simply considering the disjoint union of partitions $\pi_s \in \mcal{P}(V_s)$ for $s \in [r]$. As before, even though $\#(A_{N, \pi}) \leq \prod_{s = 1}^r \#(\mcal{B}_{N, \pi_s})$, the fact that
\[
  \lim_{N \to \infty} \frac{\#(A_{N, \pi})}{\prod_{s = 1}^r \#(\mcal{B}_{N, \pi_s})} = 1
\]
allows us to factor\small
\[
  \lim_{N \to \infty} \E\mathopen{}\Big[{\Tr}\mathopen{}\Big(\prod_{s = 1}^r \mbf{K}_Np_s(\mcal{Z}_N)\Big)\mathclose{}\Big]\mathclose{} = \lim_{N \to \infty} \prod_{s = 1}^r \frac{1}{N} \sum_{\pi_s \in \mcal{P}(V_s)} \sum_{\phi_s \in \mcal{B}_{N, \pi_s}} \E\mathopen{}\Big[\prod_{e \in E_s^{\pi_s}} \mbf{\Xi}_N^{(\gamma(e))}(\phi_s(e))\Big]\mathclose{}.
\]\normalsize
So, we will be done if we can prove that
\[
  \lim_{N \to \infty} \frac{1}{N} \sum_{\pi_s \in \mcal{P}(V_s)} \sum_{\phi_s \in \mcal{B}_{N, \pi_s}} \E\mathopen{}\Big[\prod_{e \in E_s^{\pi_s}} \mbf{\Xi}_N^{(\gamma(e))}(\phi_s(e))\Big]\mathclose{} = \tau_{\mcal{Z}}(p_s),
\]
but this follows from Lemma \ref{lem:paths_double_trees} and Example \ref{eg:traffic_probability} (recall that \cite{Au18} allows us to restrict to $\pi_s \in \mcal{P}(V_s)$ such that $T_s^{\pi_s}$ is a colored double tree).
\end{proof}

As in the case of the matrix units, assuming an infinitesimal distribution for $\mcal{Z}_N$, Lemma \ref{lem:mixed_trace_all_ones} proves that $\mcal{Z}_N$ and $\mbf{K}_N$ are asymptotically infinitesimally free. To complete the proof of Theorems \ref{thm:infinitesimal_freeness_wigner} and \ref{thm:infinitesimal_freeness_band}, we turn our attention to the non-unital algebra $\mcal{F}_N$ generated by $\mcal{E}_N$ and $\mbf{K}_N$.

\begin{lemma}\label{lem:algebra_all_ones_matrix_units}
The algebra $\mcal{F}_N$ is spanned by elements of the form
\begin{enumerate}[label=(\roman*)]
\item $\prod_{s =  1}^t (\mbf{E}_N^{(j_{s-1}, k_s)}\mbf{K}_N) = \frac{1}{N^{t-1}}\mbf{E}_N^{(j_0, j_0)}\mbf{K}_N$, where $t \geq 1$; \label{def:columns_monomial}
\item $\prod_{s =  1}^t (\mbf{K}_N\mbf{E}_N^{(j_{s-1}, k_s)}) = \frac{1}{N^{t-1}}\mbf{K}_N\mbf{E}_N^{(k_t, k_t)}$, where $t \geq 1$; \label{def:rows_monomial}
\item $\mathopen{}\big[\prod_{s =  1}^t (\mbf{E}_N^{(j_{s-1}, k_s)}\mbf{K}_N)\big]\mathclose{}\mbf{E}_N^{(j_t, k_{t+1})} = \frac{1}{N^t} \mbf{E}_N^{(j_0, k_{t+1})}$, where $t \geq 0$; \label{def:matrix_unit_monomial}
\item $\mathopen{}\big[\prod_{s =  1}^t (\mbf{K}_N\mbf{E}_N^{(j_{s-1}, k_s)})\big]\mathclose{}\mbf{K}_N= \frac{1}{N^t}\mbf{K}_N$, where $t \geq 0$. \label{def:all_ones_monomial}
\end{enumerate}
\begin{proof}
The result follows from a simple computation using the fact that $\mbf{K}_N$ is idempotent and the identity $\mbf{E}_N^{(j_0, k_1)}\mbf{E}_N^{(j_1, k_2)} = \indc{k_1 = j_1}\mbf{E}_N^{(j_0, k_2)}$.
\end{proof}
\end{lemma}

\begin{cor}\label{cor:infinitesimal_freeness_all_ones_matrix_units}
$\mcal{E}_N$ and $\mbf{K}_N$ are asymptotically infinitesimally free.
\end{cor}
\begin{proof}
The characterization of $\mcal{F}_N$ in Lemma \ref{lem:algebra_all_ones_matrix_units} implies that
\[
  {\Tr}\mathopen{}\Big(\prod_{s =  1}^t (\mbf{E}_N^{(j_{s-1}, k_s)}\mbf{K}_N)\Big)\mathclose{} = O(N^{-t}).
\]
Once again, a straightforward application of Proposition \ref{prop:criteria_infinitesimal_freeness} proves the result.
\end{proof}

We adopt the notation $\mbf{E}_N^{(0, 0)} = \mbf{K}_N$ to characterize the type $B$ distribution of $\mcal{Z}_N \cup \mcal{F}_N$. The following result implies that the only non-trivial values of $(\mu_{\mcal{Z} \cup \mcal{F}}, \nu_{\mcal{Z} \cup \mcal{F}})$ have already been computed in Lemmas \ref{lem:mixed_trace_matrix_units} and \ref{lem:mixed_trace_all_ones}.

\begin{lemma}\label{lem:mixed_trace_matrix_units_and_all_ones}
For any NC monomials $q_1, \ldots, q_r \in \C\langle\mbf{y}\rangle$ and $p_1, \ldots, p_r \in \C\langle \mbf{x}\rangle$,
\begin{equation}\label{eq:mixed_trace_high_degree}
  \lim_{N \to \infty} {\Tr}\mathopen{}\Big(\prod_{s =  1}^r q_s(\mcal{F}_N)p_s(\mcal{Z}_N)\Big)\mathclose{} = 0 \quad \text{if} \quad \sum_{s = 1}^r \deg(q_s) > r.
\end{equation}
Otherwise, $\deg(q_s) = 1$ for each $s \in [r]$, in which case $q_s(\mcal{F}_N) \in \mcal{E}_N \cup \{\mbf{E}_N^{(0, 0)}\}$ and
\begin{equation}\label{eq:mixed_trace}
  \lim_{N \to \infty} {\Tr}\mathopen{}\Big(\prod_{s =  1}^r \mbf{E}_N^{(j_{s-1}, k_s)}p_s(\mcal{Z}_N)\Big)\mathclose{} = \prod_{s = 1}^r [\indc{k_s = j_s}\tau_{\mcal{Z}}(p_s)],
\end{equation}
where $j_0 = j_r$. In particular, since the index $0$ only comes in pairs $(j_{s-1}, k_s) = (0, 0)$, the limit vanishes if there exist $s, s' \in [r]$ such that $\mbf{E}_N^{(j_{s-1}, k_s)} = \mbf{E}_N^{(0, 0)}$ and $\mbf{E}_N^{(j_{s'-1}, k_{s'})} \in \mcal{E}_N$.
\end{lemma}
\begin{proof}
The proof of \eqref{eq:mixed_trace_high_degree} will follow from our analysis of \eqref{eq:mixed_trace}, which we prove first. By our earlier work, we need only to consider the case of $s, s' \in [r]$ such that $\mbf{E}_N^{(j_{s-1}, k_s)} = \mbf{E}_N^{(0, 0)}$ and $\mbf{E}_N^{(j_{s'-1}, k_{s'})} \in \mcal{E}_N$. Moreover, the cyclic invariance of the trace allows us to assume that this occurs precisely at the values
\begin{equation}\label{eq:cross_over}
\mbf{E}_N^{(j_0, k_1)} = \mbf{E}_N^{(0, 0)} \quad \text{and} \quad \mbf{E}_N^{(j_1, k_2)} \in \mcal{E}_N.
\end{equation}
We think of each occurrence of $\mbf{E}_N^{(j_{s-1}, k_s)} = \mbf{E}_N^{(0, 0)}$ as providing its $\frac{1}{N}$ normalization to the test graph $T_s$ associated to $p_s(\mcal{Z}_N)$. Similarly, each occurrence of a matrix unit $\mbf{E}_N^{(j_{s-1}, k_s)} \in \mcal{E}_N$ creates a special vertex in each of the test graphs $T_{s-1}$ and $T_s$.

To adapt our earlier work, we define
\begin{alignat*}{3}
  J &= \mleft\{s \in [r] \mid j_s \neq 0\mright\} = \{s_J^{(1)}, \ldots, s_J^{(r_1)}\} \quad &&\text{and} \quad J^c &&= \{s_{J^c}^{(1)}, \ldots, s_{J^c}^{(r_2)}\};\\
  K &= \mleft\{s \in [r] \mid k_s \neq 0\mright\} = \{s_K^{(1)}, \ldots, s_K^{(r_1)}\} \quad &&\text{and} \quad K^c &&= \{s_{K^c}^{(1)}, \ldots, s_{K^c}^{(r_2)}\},
\end{alignat*}
where $r_1 + r_2 = r$. Similarly, we redefine
\begin{align*}
  A_{N, \pi} &= \mleft\{\phi: V^\pi \hookrightarrow [N] \mathrel{}\mathclose{}\middle|\mathopen{}\mathrel{} \begin{aligned} \phi([v_{s, 0}]_{\sim_\pi}) = k_s, \quad &\forall s \in K\\ \phi([v_{s, d_s}]_{\sim_\pi}) = j_s, \quad &\forall s \in J\\|\phi(\source(e)) - \phi(\target(e))|_N \leq b_N^{(\gamma(e))}, \quad &\forall e \in E\end{aligned}\mright\};\\
  A_{N, \pi}^{(n)} &= \mleft\{\phi_n: V_{(n)}^\pi \hookrightarrow [N] \mathrel{}\mathclose{}\middle|\mathopen{}\mathrel{} \begin{aligned}\begin{aligned} \phi_n([v_{s, 0}]_{\sim_\pi}) &= k_s & &\text{if } [v_{s, 0}]_{\sim_\pi} \in V_{(n)}^\pi \text{ and } s \in K\\ \phi_n([v_{s, d_s}]_{\sim_\pi}) &= j_s & &\text{if } [v_{s, d_s}]_{\sim_\pi} \in V_{(n)}^\pi \text{ and } s \in J\end{aligned}\\|\phi_n(\source(e)) - \phi_n(\target(e))|_N \leq b_N^{(\gamma(e))},\quad \forall e \in E_{(n)}^\pi\end{aligned}\mright\}.
\end{align*}

Note that each connected component $T_{(n)}^\pi$ of $T^\pi$ satisfies at least one of the following conditions:
\begin{enumerate}[label=(\roman*)]
\item $T_{(n)}^\pi$ has at least one special vertex with a fixed label, in which case we can apply \eqref{eq:bound_components};
\item $T_{(n)}^\pi$ contains the edges of a test graph $T_s$ that has been assigned the normalization $\frac{1}{N}$ of its adjacent term $\mbf{E}_N^{(j_{s-1}, k_s)} = \mbf{E}_N^{(0, 0)}$, in which case we can apply \eqref{eq:weak_bound_components},
\end{enumerate}
where the number $u_2$ of connected components of type (ii) satisfies $u_2 \leq r_2$. In particular, the connected component $T_{(n_*)}^\pi$ that contains the edges of $T_1$ will satisfy both of these conditions. Indeed, this follows from \eqref{eq:cross_over}. Rearranging, we can count the connected components of type (ii) first, namely, $T_{(1)}^\pi, \ldots, T_{(u_2)}^\pi$ with $u_2 = n_*$. The analogue of \eqref{eq:expectation_asymptotic} and \eqref{eq:expectation_asymptotic_all_ones} in this case then follows:
\begin{gather*}
  {\Tr}\mathopen{}\Big(\prod_{s =  1}^r \mbf{E}_N^{(j_{s-1}, k_s)}p_s(\mcal{Z}_N)\Big)\mathclose{} = \sum_{\pi \in \mcal{P}(V)} \sum_{\phi \in A_{N, \pi}} \frac{1}{N^{r_2}}\prod_{n = 1}^u \E\mathopen{}\Big[\prod_{e \in E_{(n)}^\pi} \mbf{\Xi}_N^{(\gamma(e))}(\phi(e))\Big]\mathclose{} \\
  = \sum_{\pi \in \mcal{P}(V)} \frac{1}{N^{r_2-u_2}}\prod_{n = 1}^{u_2} O_d\mathopen{}\Bigg(\frac{\#(A_{N, \pi}^{(n)})}{N\prod_{e \in E_{(n)}^\pi} \sqrt{\xi_N^{(\gamma(e))}}}\Bigg)\mathclose{} \prod_{n = u_2 + 1}^u O_d\mathopen{}\Bigg(\frac{\#(A_{N, \pi}^{(n)})}{\prod_{e \in E_{(n)}^\pi} \sqrt{\xi_N^{(\gamma(e))}}}\Bigg)\mathclose{} \\
  = \sum_{\pi \in \mcal{P}(V)} \frac{1}{N^{r_2-u_2}}\mathopen{}\Big[\prod_{n = 1}^{u_2-1} O_d(1)\Big]\mathclose{}O_d(N^{-1})\mathopen{}\Big[\prod_{n = u_2 + 1}^u O_d(1)\Big]\mathclose{} = O_d(N^{-1}),
\end{gather*}
which proves \eqref{eq:mixed_trace}.

To prove \eqref{eq:mixed_trace_high_degree}, we use the characterization of a monomial $q_s(\mcal{F}_N)$ given in Lemma \ref{lem:algebra_all_ones_matrix_units}. In particular, we imagine replacing the terms in the trace
\[
  {\Tr}\mathopen{}\Big(\prod_{s =  1}^r q_s(\mcal{F}_N)p_s(\mcal{Z}_N)\Big)\mathclose{}
\]
according to the following scheme:
\begin{enumerate}[label=(\alph*)]
\item if $q_s(\mcal{F}_N)$ is of the form \ref{def:columns_monomial} or \ref{def:rows_monomial}, then we replace $q_s(\mcal{F}_N)$ with $\mbf{E}_N^{(0, 0)}$; \label{def:replace_double_normalization}
\item if $q_s(\mcal{F}_N)$ is of the form \ref{def:matrix_unit_monomial}, then we replace $q_s(\mcal{F}_N)$ with the corresponding matrix unit without the factor of $\frac{1}{N^t}$; \label{def:replace_matrix_unit_power}
\item if $q_s(\mcal{F}_N)$ is of the form \ref{def:all_ones_monomial}, then we replace $q_s(\mcal{F}_N)$ with $\mbf{E}_N^{(0, 0)}$ without the factor of $\frac{1}{N^t}$. \label{def:replace_all_ones_power}
\end{enumerate}
After this procedure, our work above shows that the resulting trace satisfies
\begin{equation}\label{eq:bounded_trace}
{\Tr}\mathopen{}\Big(\prod_{s =  1}^r \mbf{E}_N^{(j_{s-1}, k_s)}p_s(\mcal{Z}_N)\Big)\mathclose{} = O_d(1);
\end{equation}
however, based on our analysis of $\#(A_{N, \pi})$, the original trace then necessarily satisfies
\begin{equation}\label{eq:high_degree_asymptotic}
  {\Tr}\mathopen{}\Big(\prod_{s =  1}^r q_s(\mcal{F}_N)p_s(\mcal{Z}_N)\Big)\mathclose{} = o_d(1).
\end{equation}
Indeed, consider the following interpretation of the replacement scheme. If the original term is of type \ref{def:columns_monomial} (resp., type \ref{def:rows_monomial}), then it creates a special vertex in the test graph $T_{s-1}$ (\emph{resp.,} $T_s$) and contributes a factor of $\frac{1}{N^t}$ to the test graph $T_s$, where $t = \frac{\deg(q_s)}{2}$. In contrast, its replacement $\mbf{E}_N^{(0, 0)}$ only contributes a factor of $\frac{1}{N}$ to $T_s$. Similarly, if the original term is of type \ref{def:matrix_unit_monomial} or type \ref{def:all_ones_monomial}, then its replacement simply drops the factor of $\frac{1}{N^t}$, where $t = \frac{\deg(q_s) - 1}{2}$. In any case, since $\sum_{s=1}^r \deg(q_s) > r$, we know that a replacement of type \ref{def:replace_double_normalization}, \ref{def:replace_matrix_unit_power}, or \ref{def:replace_all_ones_power} occurs with $t \geq 1$. Our work in establishing \eqref{eq:bounded_trace} then proves \eqref{eq:high_degree_asymptotic}. The result now follows.
\end{proof}

\begin{cor}\label{cor:infinitesimally_free_band_matrices}
Assume that the family $\mcal{Z}_N$ has an infinitesimal distribution. Then the matrices $\mcal{Z}_N$, $\mcal{E}_N$, and $\mbf{K}_N$ are asymptotically infinitesimally free.
\end{cor}
\begin{proof}
Under the assumption for $\mcal{Z}_N$, we already know that each pair of the families $\mcal{Z}_N$, $\mcal{E}_N$, and $\mbf{K}_N$ are asymptotically infinitesimally free by Lemmas \ref{lem:mixed_trace_matrix_units} and \ref{lem:mixed_trace_all_ones} and Corollary \ref{cor:infinitesimal_freeness_all_ones_matrix_units}. So, we will be done if we can prove that $\mcal{Z}_N$ and $\mcal{F}_N$ are asymptotically infinitesimally free. Again, this follows from applying the criteria in Proposition \ref{prop:criteria_infinitesimal_freeness} to Lemma \ref{lem:mixed_trace_matrix_units_and_all_ones}.
\end{proof}

This completes the proof of Theorems \ref{thm:infinitesimal_freeness_wigner} and \ref{thm:infinitesimal_freeness_band}. The type $B$ free convolution calculations in Corollaries \ref{cor:outliers_wigner} and \ref{cor:outliers_band} essentially already appear in \cite[\S 4.1.1]{Shl18}, so we do not repeat them. 

\addtocontents{toc}{\SkipTocEntry}
\subsection*{Acknowledgements}
The author thanks Paul Bourgade, James Mingo, and Dimitri Shlyakhtenko for helpful conversations. The figures in this article were produced in Inkscape.

\bibliographystyle{amsalpha}
\bibliography{band_infinitesimal}

\end{document}

%% file: fig1_hist_localized.pdf_tex
\begingroup%
  \makeatletter%
  \providecommand\color[2][]{%
    \errmessage{(Inkscape) Color is used for the text in Inkscape, but the package 'color.sty' is not loaded}%
    \renewcommand\color[2][]{}%
  }%
  \providecommand\transparent[1]{%
    \errmessage{(Inkscape) Transparency is used (non-zero) for the text in Inkscape, but the package 'transparent.sty' is not loaded}%
    \renewcommand\transparent[1]{}%
  }%
  \providecommand\rotatebox[2]{#2}%
  \newcommand*\fsize{\dimexpr\f@size pt\relax}%
  \newcommand*\lineheight[1]{\fontsize{\fsize}{#1\fsize}\selectfont}%
  \ifx\svgwidth\undefined%
    \setlength{\unitlength}{360bp}%
    \ifx\svgscale\undefined%
      \relax%
    \else%
      \setlength{\unitlength}{\unitlength * \real{\svgscale}}%
    \fi%
  \else%
    \setlength{\unitlength}{\svgwidth}%
  \fi%
  \global\let\svgwidth\undefined%
  \global\let\svgscale\undefined%
  \makeatother%
  \begin{picture}(1,0.6)%
    \lineheight{1}%
    \setlength\tabcolsep{0pt}%
    \put(0.12000119,0.0136559){\makebox(0,0)[t]{\lineheight{1.25}\smash{\begin{tabular}[t]{c}-3\end{tabular}}}}%
    \put(0.2527571,0.0136559){\makebox(0,0)[t]{\lineheight{1.25}\smash{\begin{tabular}[t]{c}-2\end{tabular}}}}%
    \put(0.385513,0.0136559){\makebox(0,0)[t]{\lineheight{1.25}\smash{\begin{tabular}[t]{c}-1\end{tabular}}}}%
    \put(0.51826891,0.0136559){\makebox(0,0)[t]{\lineheight{1.25}\smash{\begin{tabular}[t]{c}0\end{tabular}}}}%
    \put(0.65110355,0.0136559){\makebox(0,0)[t]{\lineheight{1.25}\smash{\begin{tabular}[t]{c}1\end{tabular}}}}%
    \put(0.78385946,0.0136559){\makebox(0,0)[t]{\lineheight{1.25}\smash{\begin{tabular}[t]{c}2\end{tabular}}}}%
    \put(0.91661536,0.0136559){\makebox(0,0)[t]{\lineheight{1.25}\smash{\begin{tabular}[t]{c}3\end{tabular}}}}%
    \put(0,0){\includegraphics[width=\unitlength,page=1]{fig1_hist_localized.pdf}}%
    \put(0.09637914,0.04841721){\makebox(0,0)[rt]{\lineheight{1.25}\smash{\begin{tabular}[t]{r}0.0\end{tabular}}}}%
    \put(0.09637914,0.14385029){\makebox(0,0)[rt]{\lineheight{1.25}\smash{\begin{tabular}[t]{r}0.1\end{tabular}}}}%
    \put(0.09637914,0.23928336){\makebox(0,0)[rt]{\lineheight{1.25}\smash{\begin{tabular}[t]{r}0.2\end{tabular}}}}%
    \put(0.09637914,0.33471643){\makebox(0,0)[rt]{\lineheight{1.25}\smash{\begin{tabular}[t]{r}0.3\end{tabular}}}}%
    \put(0.09637914,0.43022824){\makebox(0,0)[rt]{\lineheight{1.25}\smash{\begin{tabular}[t]{r}0.4\end{tabular}}}}%
    \put(0.09637914,0.52566131){\makebox(0,0)[rt]{\lineheight{1.25}\smash{\begin{tabular}[t]{r}0.5\end{tabular}}}}%
    \put(0.78616473,0.49596365){\makebox(0,0)[lt]{\lineheight{1.25}\smash{\begin{tabular}[t]{l}$b_N=N^{3/5}$\end{tabular}}}}%
    \put(0.78616552,0.46471365){\makebox(0,0)[lt]{\lineheight{1.25}\smash{\begin{tabular}[t]{l}$b_N=N^{2/5}$\end{tabular}}}}%
    \put(0,0){\includegraphics[width=\unitlength,page=2]{fig1_hist_localized.pdf}}%
  \end{picture}%
\endgroup%

%% file: fig2_hist_delocalized.pdf_tex
\begingroup%
  \makeatletter%
  \providecommand\color[2][]{%
    \errmessage{(Inkscape) Color is used for the text in Inkscape, but the package 'color.sty' is not loaded}%
    \renewcommand\color[2][]{}%
  }%
  \providecommand\transparent[1]{%
    \errmessage{(Inkscape) Transparency is used (non-zero) for the text in Inkscape, but the package 'transparent.sty' is not loaded}%
    \renewcommand\transparent[1]{}%
  }%
  \providecommand\rotatebox[2]{#2}%
  \newcommand*\fsize{\dimexpr\f@size pt\relax}%
  \newcommand*\lineheight[1]{\fontsize{\fsize}{#1\fsize}\selectfont}%
  \ifx\svgwidth\undefined%
    \setlength{\unitlength}{360bp}%
    \ifx\svgscale\undefined%
      \relax%
    \else%
      \setlength{\unitlength}{\unitlength * \real{\svgscale}}%
    \fi%
  \else%
    \setlength{\unitlength}{\svgwidth}%
  \fi%
  \global\let\svgwidth\undefined%
  \global\let\svgscale\undefined%
  \makeatother%
  \begin{picture}(1,0.65999999)%
    \lineheight{1}%
    \setlength\tabcolsep{0pt}%
    \put(0.12000119,0.01365591){\makebox(0,0)[t]{\lineheight{1.25}\smash{\begin{tabular}[t]{c}-3\end{tabular}}}}%
    \put(0.25275709,0.01365591){\makebox(0,0)[t]{\lineheight{1.25}\smash{\begin{tabular}[t]{c}-2\end{tabular}}}}%
    \put(0.385513,0.01365591){\makebox(0,0)[t]{\lineheight{1.25}\smash{\begin{tabular}[t]{c}-1\end{tabular}}}}%
    \put(0.5182689,0.01365591){\makebox(0,0)[t]{\lineheight{1.25}\smash{\begin{tabular}[t]{c}0\end{tabular}}}}%
    \put(0.65110354,0.01365591){\makebox(0,0)[t]{\lineheight{1.25}\smash{\begin{tabular}[t]{c}1\end{tabular}}}}%
    \put(0.78385945,0.01365591){\makebox(0,0)[t]{\lineheight{1.25}\smash{\begin{tabular}[t]{c}2\end{tabular}}}}%
    \put(0.91661535,0.01365591){\makebox(0,0)[t]{\lineheight{1.25}\smash{\begin{tabular}[t]{c}3\end{tabular}}}}%
    \put(0,0){\includegraphics[width=\unitlength,page=1]{fig2_hist_delocalized.pdf}}%
    \put(0.09637914,0.04841723){\makebox(0,0)[rt]{\lineheight{1.25}\smash{\begin{tabular}[t]{r}0.0\end{tabular}}}}%
    \put(0.09637914,0.1438503){\makebox(0,0)[rt]{\lineheight{1.25}\smash{\begin{tabular}[t]{r}0.1\end{tabular}}}}%
    \put(0.09637914,0.23928337){\makebox(0,0)[rt]{\lineheight{1.25}\smash{\begin{tabular}[t]{r}0.2\end{tabular}}}}%
    \put(0.09637914,0.33471644){\makebox(0,0)[rt]{\lineheight{1.25}\smash{\begin{tabular}[t]{r}0.3\end{tabular}}}}%
    \put(0.09637914,0.43022825){\makebox(0,0)[rt]{\lineheight{1.25}\smash{\begin{tabular}[t]{r}0.4\end{tabular}}}}%
    \put(0.09637914,0.52566132){\makebox(0,0)[rt]{\lineheight{1.25}\smash{\begin{tabular}[t]{r}0.5\end{tabular}}}}%
    \put(0.78616447,0.49596365){\makebox(0,0)[lt]{\lineheight{1.25}\smash{\begin{tabular}[t]{l}$b_N=N^{3/5}$\end{tabular}}}}%
    \put(0.78616447,0.46471365){\makebox(0,0)[lt]{\lineheight{1.25}\smash{\begin{tabular}[t]{l}$b_N=N^{2/5}$\end{tabular}}}}%
    \put(0,0){\includegraphics[width=\unitlength,page=2]{fig2_hist_delocalized.pdf}}%
  \end{picture}%
\endgroup%

%% file: fig3_qq_localized.pdf_tex
\begingroup%
  \makeatletter%
  \providecommand\color[2][]{%
    \errmessage{(Inkscape) Color is used for the text in Inkscape, but the package 'color.sty' is not loaded}%
    \renewcommand\color[2][]{}%
  }%
  \providecommand\transparent[1]{%
    \errmessage{(Inkscape) Transparency is used (non-zero) for the text in Inkscape, but the package 'transparent.sty' is not loaded}%
    \renewcommand\transparent[1]{}%
  }%
  \providecommand\rotatebox[2]{#2}%
  \newcommand*\fsize{\dimexpr\f@size pt\relax}%
  \newcommand*\lineheight[1]{\fontsize{\fsize}{#1\fsize}\selectfont}%
  \ifx\svgwidth\undefined%
    \setlength{\unitlength}{360bp}%
    \ifx\svgscale\undefined%
      \relax%
    \else%
      \setlength{\unitlength}{\unitlength * \real{\svgscale}}%
    \fi%
  \else%
    \setlength{\unitlength}{\svgwidth}%
  \fi%
  \global\let\svgwidth\undefined%
  \global\let\svgscale\undefined%
  \makeatother%
  \begin{picture}(1,0.6)%
    \lineheight{1}%
    \setlength\tabcolsep{0pt}%
    \put(0.10953481,0.01365591){\makebox(0,0)[t]{\lineheight{1.25}\smash{\begin{tabular}[t]{c}-2\end{tabular}}}}%
    \put(0.31110962,0.01365591){\makebox(0,0)[t]{\lineheight{1.25}\smash{\begin{tabular}[t]{c}-1\end{tabular}}}}%
    \put(0.51276316,0.01365591){\makebox(0,0)[t]{\lineheight{1.25}\smash{\begin{tabular}[t]{c}0\end{tabular}}}}%
    \put(0.71433796,0.01365591){\makebox(0,0)[t]{\lineheight{1.25}\smash{\begin{tabular}[t]{c}1\end{tabular}}}}%
    \put(0.91591277,0.01365591){\makebox(0,0)[t]{\lineheight{1.25}\smash{\begin{tabular}[t]{c}2\end{tabular}}}}%
    \put(0,0){\includegraphics[width=\unitlength,page=1]{fig3_qq_localized.pdf}}%
    \put(0.08591277,0.04808415){\makebox(0,0)[rt]{\lineheight{1.25}\smash{\begin{tabular}[t]{r}-2\end{tabular}}}}%
    \put(0.08591277,0.16737549){\makebox(0,0)[rt]{\lineheight{1.25}\smash{\begin{tabular}[t]{r}-1\end{tabular}}}}%
    \put(0.08591277,0.28666683){\makebox(0,0)[rt]{\lineheight{1.25}\smash{\begin{tabular}[t]{r}0\end{tabular}}}}%
    \put(0.08591277,0.40603691){\makebox(0,0)[rt]{\lineheight{1.25}\smash{\begin{tabular}[t]{r}1\end{tabular}}}}%
    \put(0.08591277,0.52532825){\makebox(0,0)[rt]{\lineheight{1.25}\smash{\begin{tabular}[t]{r}2\end{tabular}}}}%
    \put(0,0){\includegraphics[width=\unitlength,page=2]{fig3_qq_localized.pdf}}%
    \put(0.76439362,0.12453787){\makebox(0,0)[lt]{\lineheight{1.25}\smash{\begin{tabular}[t]{l}$b_N=N^{3/5}$\end{tabular}}}}%
    \put(0.7643944,0.09328787){\makebox(0,0)[lt]{\lineheight{1.25}\smash{\begin{tabular}[t]{l}$b_N=N^{2/5}$\end{tabular}}}}%
    \put(0,0){\includegraphics[width=\unitlength,page=3]{fig3_qq_localized.pdf}}%
  \end{picture}%
\endgroup%

%% file: fig4_qq_delocalized.pdf_tex
\begingroup%
  \makeatletter%
  \providecommand\color[2][]{%
    \errmessage{(Inkscape) Color is used for the text in Inkscape, but the package 'color.sty' is not loaded}%
    \renewcommand\color[2][]{}%
  }%
  \providecommand\transparent[1]{%
    \errmessage{(Inkscape) Transparency is used (non-zero) for the text in Inkscape, but the package 'transparent.sty' is not loaded}%
    \renewcommand\transparent[1]{}%
  }%
  \providecommand\rotatebox[2]{#2}%
  \newcommand*\fsize{\dimexpr\f@size pt\relax}%
  \newcommand*\lineheight[1]{\fontsize{\fsize}{#1\fsize}\selectfont}%
  \ifx\svgwidth\undefined%
    \setlength{\unitlength}{360bp}%
    \ifx\svgscale\undefined%
      \relax%
    \else%
      \setlength{\unitlength}{\unitlength * \real{\svgscale}}%
    \fi%
  \else%
    \setlength{\unitlength}{\svgwidth}%
  \fi%
  \global\let\svgwidth\undefined%
  \global\let\svgscale\undefined%
  \makeatother%
  \begin{picture}(1,0.65999999)%
    \lineheight{1}%
    \setlength\tabcolsep{0pt}%
    \put(0.10953481,0.0136559){\makebox(0,0)[t]{\lineheight{1.25}\smash{\begin{tabular}[t]{c}-2\end{tabular}}}}%
    \put(0.31110961,0.0136559){\makebox(0,0)[t]{\lineheight{1.25}\smash{\begin{tabular}[t]{c}-1\end{tabular}}}}%
    \put(0.51276315,0.0136559){\makebox(0,0)[t]{\lineheight{1.25}\smash{\begin{tabular}[t]{c}0\end{tabular}}}}%
    \put(0.71433795,0.0136559){\makebox(0,0)[t]{\lineheight{1.25}\smash{\begin{tabular}[t]{c}1\end{tabular}}}}%
    \put(0.91591275,0.0136559){\makebox(0,0)[t]{\lineheight{1.25}\smash{\begin{tabular}[t]{c}2\end{tabular}}}}%
    \put(0,0){\includegraphics[width=\unitlength,page=1]{fig4_qq_delocalized.pdf}}%
    \put(0.085913,0.04808416){\makebox(0,0)[rt]{\lineheight{1.25}\smash{\begin{tabular}[t]{r}-2\end{tabular}}}}%
    \put(0.085913,0.16737549){\makebox(0,0)[rt]{\lineheight{1.25}\smash{\begin{tabular}[t]{r}-1\end{tabular}}}}%
    \put(0.085913,0.28666683){\makebox(0,0)[rt]{\lineheight{1.25}\smash{\begin{tabular}[t]{r}0\end{tabular}}}}%
    \put(0.085913,0.40603691){\makebox(0,0)[rt]{\lineheight{1.25}\smash{\begin{tabular}[t]{r}1\end{tabular}}}}%
    \put(0.085913,0.52532824){\makebox(0,0)[rt]{\lineheight{1.25}\smash{\begin{tabular}[t]{r}2\end{tabular}}}}%
    \put(0,0){\includegraphics[width=\unitlength,page=2]{fig4_qq_delocalized.pdf}}%
    \put(0.76439359,0.12335115){\makebox(0,0)[lt]{\lineheight{1.25}\smash{\begin{tabular}[t]{l}$b_N=N^{3/5}$\end{tabular}}}}%
    \put(0.76439438,0.09210115){\makebox(0,0)[lt]{\lineheight{1.25}\smash{\begin{tabular}[t]{l}$b_N=N^{2/5}$\end{tabular}}}}%
  \end{picture}%
\endgroup%

%% file: fig5_genus_graph.pdf_tex
\begingroup%
  \makeatletter%
  \providecommand\color[2][]{%
    \errmessage{(Inkscape) Color is used for the text in Inkscape, but the package 'color.sty' is not loaded}%
    \renewcommand\color[2][]{}%
  }%
  \providecommand\transparent[1]{%
    \errmessage{(Inkscape) Transparency is used (non-zero) for the text in Inkscape, but the package 'transparent.sty' is not loaded}%
    \renewcommand\transparent[1]{}%
  }%
  \providecommand\rotatebox[2]{#2}%
  \newcommand*\fsize{\dimexpr\f@size pt\relax}%
  \newcommand*\lineheight[1]{\fontsize{\fsize}{#1\fsize}\selectfont}%
  \ifx\svgwidth\undefined%
    \setlength{\unitlength}{360bp}%
    \ifx\svgscale\undefined%
      \relax%
    \else%
      \setlength{\unitlength}{\unitlength * \real{\svgscale}}%
    \fi%
  \else%
    \setlength{\unitlength}{\svgwidth}%
  \fi%
  \global\let\svgwidth\undefined%
  \global\let\svgscale\undefined%
  \makeatother%
  \begin{picture}(1,0.25)%
    \lineheight{1}%
    \setlength\tabcolsep{0pt}%
    \put(0.86855593,0.00786136){\color[rgb]{0,0,0}\makebox(0,0)[lt]{\lineheight{1.25}\smash{\begin{tabular}[t]{l}$\underline{C}_{2\ell}^\pi$\end{tabular}}}}%
    \put(0,0){\includegraphics[width=\unitlength,page=1]{fig5_genus_graph.pdf}}%
    \put(0.15913481,0.23209516){\color[rgb]{0,0,0}\makebox(0,0)[lt]{\lineheight{1.25}\smash{\begin{tabular}[t]{l}$e_1$\end{tabular}}}}%
    \put(0,0){\includegraphics[width=\unitlength,page=2]{fig5_genus_graph.pdf}}%
    \put(0.54056129,0.00786136){\color[rgb]{0,0,0}\makebox(0,0)[lt]{\lineheight{1.25}\smash{\begin{tabular}[t]{l}$C_{2\ell}^\pi$\end{tabular}}}}%
    \put(0.61582773,0.16091759){\color[rgb]{0,0,0}\makebox(0,0)[lt]{\lineheight{1.25}\smash{\begin{tabular}[t]{l}$e_1$\end{tabular}}}}%
    \put(0.38406522,0.10275349){\color[rgb]{0,0,0}\makebox(0,0)[lt]{\lineheight{1.25}\smash{\begin{tabular}[t]{l}$\mapsto$\end{tabular}}}}%
    \put(0.71205734,0.10275253){\color[rgb]{0,0,0}\makebox(0,0)[lt]{\lineheight{1.25}\smash{\begin{tabular}[t]{l}$\mapsto$\end{tabular}}}}%
  \end{picture}%
\endgroup%